\documentclass[11pt]{article}
\usepackage[T1]{fontenc}
\usepackage[latin1]{inputenc}
\usepackage[margin=0.8in]{geometry}
\usepackage{amsmath,amssymb,amstext}
\usepackage{amsthm}
\usepackage{hyperref}
\usepackage{cleveref}
\usepackage{authblk}
\usepackage{enumerate}
\usepackage{dirtytalk}
\usepackage{csquotes}
\usepackage{tikz}
\usepackage{pifont}
\usepackage{thmtools}
\usepackage{thm-restate}
\usepackage{soul}
\usetikzlibrary{shapes}
\usepackage{array}
\usepackage{float}
\usepackage{diagbox}

\usepackage{paralist}

\usepackage{tikz}
\usepackage{thmtools}
\usepackage{thm-restate}
\usetikzlibrary{shapes}

\newtheorem{thm}{Theorem}[section]

\newtheorem{lemma}[thm]{Lemma}
\newtheorem{conj}[thm]{Conjecture}

\newtheorem{obs}[thm]{Observation}

\numberwithin{equation}{section}
\theoremstyle{definition} 
\newtheorem{definition}[thm]{Definition}

\definecolor{nicelavender}{RGB}{153, 128, 250}

\newcommand\ew{{\rm ew}_{2}}
\renewcommand{\phi}{\varphi}
\definecolor{nicered}{RGB}{248, 0, 28}

\definecolor{asparagus}{rgb}{0.53, 0.66, 0.42}

\definecolor{cerulean}{rgb}{0.0, 0.48, 0.65}
\definecolor{cornellred}{rgb}{0.7, 0.11, 0.11}
\definecolor{darklavender}{rgb}{0.45, 0.31, 0.59}
\definecolor{darkslateblue}{rgb}{0.28, 0.24, 0.55}
\definecolor{burntorange}{rgb}{0.8, 0.33, 0.0}

\tikzstyle{every node}=[circle, draw, fill=black!50,
                        inner sep=0pt, minimum width=4pt]
\tikzstyle{E1}=[thick]
\tikzstyle{E2}=[very thick, dashed]
\tikzstyle{deleted}=[dotted]
\tikzstyle{deletedE2}=[dash pattern= on .5pt off 5pt]
\tikzstyle{tight}=[shape=diamond, minimum width=8pt, minimum height=8pt]

\date{}

\begin{document}


\author{Luke Postle\thanks{We acknowledge the support of the Natural Sciences and Engineering Research Council of Canada (NSERC) [Discovery Grant No.  2019-04304].\\
\hphantom{m} $^*$Cette recherche a \'{e}t\'{e} financ\'{e}e par le Conseil de recherches en sciences naturelles et en g\'{e}nie du Canada (CRSNG) [Discovery Grant No.  2019-04304].} }
\author{Evelyne Smith-Roberge$^\dagger$}
\author{Massimo Vicenzo$^\flat$}

\affil{$^{*,\flat}$Dept.~of Combinatorics and Optimization, University of Waterloo \\ \texttt{\{$^*$lpostle, $^\flat$mvicenzo\}@uwaterloo.ca}}
\affil{$^\dagger$School of Mathematics, Georgia Institute of Technology \\ \texttt{esmithroberge3@gatech.edu}}
\title{Acyclic List Colouring Locally Planar Graphs}
\date{\today}
\maketitle
\date{}

\begin{abstract}
A (vertex) colouring of graph is \emph{acyclic} if it contains no bicoloured cycle. In 1979, Borodin proved that planar graphs are acyclically 5-colourable. In 2010, Kawarabayashi and Mohar proved that locally planar graphs are acyclically 7-colourable. In 2002, Borodin, Fon-Der-Flaass, Kostochka, Raspaud, and Sopena proved that planar graphs are acyclically 7-list-colourable. We prove that locally planar graphs are acyclically 9-list-colourable\textemdash no bound for acyclic list colouring locally planar graphs for any fixed number of colours was previously known.
\end{abstract}
\maketitle

\section{Introduction}
A \emph{$k$-colouring} of a graph $G$ is a function $\varphi:V(G) \rightarrow [k]$ such that for each edge $uv\in E(G)$, we require  that $\varphi(u) \neq \varphi(v)$. A graph is \emph{$k$-colourable} if it has a $k$-colouring. The \emph{chromatic number} of $G$ is denoted $\chi(G)$, and defined as the smallest number $k$ such that $G$ is $k$-colourable.  

Given a class of graphs, the problem of determining the chromatic number of the class (i.e.~the maximum number $k$ such that there exists a graph in the class with chromatic number $k$) is perhaps the central question in the field of graph colouring. The Four Colour Theorem, due to Appel and Haken \cite{kenneth1977every,appel1977every}, is one of the most celebrated results in the field, and does exactly this for the class of planar graphs, showing planar graphs are 4-colourable. Recall that a graph is \emph{planar} if it can be embedded in the sphere in such a way that the edges meet only at common endpoints.

\emph{List colouring} is a generalization of vertex colouring wherein the possible images of the colouring function are local to each vertex. More formally: given a graph $G$, a \emph{$k$-list assignment} for $G$ is a function $L$ that assigns to each vertex $v \in V(G)$ a set $L(v)$ of size at least $k$. Given a list assignment $L$, an \emph{$L$-colouring} is a colouring $\varphi$ with the additional property that for each vertex $v$, we have that $\varphi(v) \in L(v)$. We say a graph is \emph{$k$-list-colourable} if it admits an $L$-colouring for every $k$-list assignment $L$.

In \cite{voigt1993list}, Voigt showed that the Four Colour Theorem does not hold in the realm of list colouring by constructing a planar graph $G$ and corresponding 4-list assignment $L$ such that $G$ is not $L$-colourable. Thomassen \cite{thomassen1994list} proved via a remarkably short and elegant argument that lists of size four are \emph{just} insufficient: every planar graph is 5-list-colourable.

This paper investigates the related question of colouring graphs embedded in surfaces other than the sphere. For the more basic surface definitions not covered here (e.g.~surface, handle, crosscap), we refer the reader to \cite{mohar2001graphsonsurfaces}. First recall the Classification of Surfaces Theorem, which states that every surface can be obtained from the sphere by either adding a set of handles, or by adding a nonempty set of crosscaps. If a surface $\Sigma$ is obtained from the sphere by adding $k$ handles, its \emph{Euler genus} is defined as $g(\Sigma) := 2k$. If $\Sigma$ is obtained by adding $k$ crosscaps, its Euler genus is $g(\Sigma) := k$. For the remainder of the paper, we refer to the Euler genus of a surface as simply its \emph{genus}. A \emph{non-contractible cycle} in a surface is a cycle that cannot be continuously deformed to a single point. We say a graph embedded in a surface is \emph{$\rho$-locally planar} if every cycle in the graph that is non-contractible has length at least $\rho$. This is closely related to the concept of \emph{edge-width}, which is the length of a shortest non-contractible cycle of an embedded graph: a $\rho$-locally planar has edge-width at least $\rho$. We use $ew(G)$ to denote the edge-width of an embedded graph $G$. The \say{locally planar} nomenclature is of course very natural: given a $\rho$-locally planar graph $G$ in a surface $\Sigma$, by taking any vertex $v$ of $G$ and the vertices $X$ within distance at most $(\frac{\rho}{2}-1)$ of $v$, the inherited embedding of the graph $G[X \cup \{v\}]$ is planar. 

Locally planar embedded graphs often have colouring properties similar in spirit to planar graphs. To contextualize this in comparison to other embedded graphs:  Heawood \cite{heawood1890map} famously proved in 1890 that if $G$ is a graph embedded in a surface $\Sigma$ other than the sphere, then $\chi(G) \leq \left \lfloor \frac{7 + \sqrt{24 g(\Sigma)+1}}{2} \right \rfloor$. In 1968, it was shown by Ringel and Youngs \cite{ringel1968solution} that this is best possible unless $\Sigma$ is the Klein bottle, which as proved by Franklin \cite{franklin1934six} requires exactly six colours. In comparison, Thomassen \cite{thomassen1993five} showed that locally planar graphs are 5-colourable by showing that if $G$ is embedded in a surface $\Sigma$ and $ew(G) \geq  2^{\Omega(g(\Sigma))}$, then $\chi(G) \leq 5$. This is also best possible, as shown by Thomassen \cite{thomassen199767lowerbd} (extending a construction of Fisk \cite{fisk1977298construction}).

When it comes to list colouring,  Devos, Kawarabayashi, and Mohar \cite{devos2006locally} showed that locally planar graphs are 5-list colourable: in particular, they showed that for every surface $\Sigma$ there exists $\rho = 2^{\Omega(g(\Sigma))}$ such that every graph $G$ embedded in $\Sigma$ with edge-width $\rho$ is 5-list colourable. This generalizes Thomassen's theorem (\cite{thomassen1993five}) regarding the 5-colourability of such graphs; the edge-width bound for list colouring (and hence also for ordinary colouring) was later improved to the asymptotically best possible value of $\rho = \Omega(\log{g(\Sigma))}$ by Postle and Thomas \cite{postle2016five}.

Table \ref{tab:non-acyclic} summarizes the best-known bounds for the number of colours required for (list) colouring planar and locally planar graphs.

\begin{table}[H]
    \centering
    \renewcommand{\arraystretch}{1.3}
    \begin{tabular}{|>{\centering}p{5.5cm}!{\vrule width 2pt}>{\centering}p{3cm}|>{\centering\arraybackslash} p{3cm}|}
        \hline
        \backslashbox[6cm]{Graph Class}{Chromatic Numbers} &\makebox[3em]{Ordinary}&\makebox[3em]{List} \\
        \noalign{\hrule height 2pt}
        Planar & $= 4$ \hskip0.7cm\cite{appel1977every} & \hskip0.1cm$= 5$ \hskip0.6cm \cite{thomassen1994list} \\
        \hline
        Locally Planar & \hskip0.1cm$=5$ \hskip0.6cm \cite{thomassen1993five} & $=5$ \hskip0.7cm \cite{devos2006locally} \\
        \hline
    \end{tabular}
    \caption{The best known bounds on the chromatic and list chromatic numbers of planar and locally planar graphs.}
    \label{tab:non-acyclic}
\end{table}

We now turn our attentions to \emph{acyclic} colouring. First, note that a $k$-colouring can be thought of as a labelling where, for each colour $i \in [k]$, we impose a restriction on the set of vertices coloured $i$: in particular, we ask that these vertices form an independent set. We say a colouring $\varphi$ of $G$ is \emph{acyclic} if for every cycle $C \subseteq G$, we have $|\{\varphi(v): v \in V(C)\}| \geq 3$. A graph is \emph{acyclically $k$-colourable} if it has an acyclic $k$-colouring. An acyclic $k$-colouring is therefore a $k$-colouring with, for each pair of distinct colours $i,j \in [k]$, an additional restriction on the set of vertices coloured $i$ or $j$: in particular, we ask that these vertices induce a forest. The \emph{acyclic chromatic number} of $G$ is denoted $\chi_a(G)$, and defined as the least $k$ such that $G$ is acyclically $k$-colourable. Clearly, as acyclic $k$-colouring is a more restrictive notion than $k$-colouring, $\chi_a(G) \geq \chi(G)$ for every graph $G$.

Acyclic colouring was first introduced by Gr\"{u}nbaum \cite{grunbaum1973acyclic9}, whose original paper on the subject has accrued over 500 citations at the time of writing. While acyclic colouring is interesting in its own right, Gr\"{u}nbaum studies this concept in the setting of planar graphs to generalize the concept of "point-arboricity" of a graph\textemdash the decomposition of a graph into vertex-disjoint acyclic subgraphs.

Unlike ordinary colouring, the Four Colour Theorem does not hold for acyclic colouring: this was demonstrated by a construction of Gr\"{u}nbaum \cite{grunbaum1973acyclic9}, who also conjectured in that $\chi_a(G) \leq 5$ if $G$ is planar.  In the same paper, Gr\"{u}nbaum showed that if $G$ is planar, $\chi_a(G) \leq 9$. This bound was improved by Mitchem \cite{mitchem1974acyclic8} to 8; and later, to 7 by Albertson and Berman \cite{albertson1977acyclic7}. Kostochka \cite{kostochka1976acyclic6} showed $\chi_a(G) \leq 6$ for every planar graph $G$; and in \cite{borodin1979acyclic5}, Borodin settled the question of the acyclic chromatic number of the class of planar graphs, showing every planar graph is acyclically 5-colourable. As mentioned prior, this is best possible.

\emph{Acyclic list colouring} is defined analogously to its non-list counterpart: an $L$-colouring is acyclic if the vertices in every cycle together use at least three colours, and a graph is \emph{acyclically-$k$-list-colourable} if it has an acyclic $L$-colouring for every $k$-list assignment $L$. Recall that all planar graphs are $5$-list colourable. This bound is also conjectured \cite{borodin2002acycliclist7} to hold for acyclic list colouring; the best known bound is due to Borodin, Fon-Der-Flaass, Kostochka, Raspaud, and Sopena \cite{borodin2002acycliclist7}, who showed the following.

\begin{thm}[Borodin et al., \cite{borodin2002acycliclist7}]\label{thm:borodin}
    Every planar graph is acyclically 7-list-colourable. 
\end{thm}

For acyclic colouring graphs on surfaces, analogous to Heawood's theorem Alon, Mohar, and Sanders \cite{alon1996acyclic} proved that if $G$ is a graph embedded in a surface $\Sigma$, then $\chi_a(G) \leq 100\cdot g(\Sigma)^\frac{4}{7} + 10 000$, and moreover $\chi_a(G) = \Omega\left(\frac{g(\Sigma)^{\frac{4}{7}}}{\log{g(\Sigma)}^{\frac{1}{7}}}\right)$.  In 2005, Mohar \cite{mohar2005acyclic} showed that locally planar graphs are acyclically 8-colourable; in fact, Mohar showed this holds in the more restrictive setting of embedded graphs with no \emph{short surface nonseparating curves}, i.e., graphs with large \emph{non-separating edge-width}. Kawarabayashi and Mohar \cite{kawarabayashi2010star} later proved that locally planar graphs are acyclically 7-colourable. (We note that it does not seem like the methods used in \cite{kawarabayashi2010star} can be used to prove an analogous result for list colouring. For instance, in certain parts of the proof, three colours are reserved for use in a specific subgraph $H$ of the graph. In the list colouring framework, if our lists are big enough we can also avoid colouring parts of the graph using a reserved set of colours. However, there is no guarantee that the vertices in $H$ contain the reserved set of colours in their lists, and so no guarantee that the colouring of the graph can be completed using this reserved colour set.)

With respect to acyclic list colouring graphs on surfaces, not much is known. This is the topic of our paper. Our main result is the following theorem, which to our knowledge is the first that gives a bound on the acyclic list chromatic number of locally planar graphs.

\begin{thm}\label{thm:main}
    For every surface $\Sigma$, there exists $\rho_0 := \rho_0(g(\Sigma))$ such that every $\rho_0$-locally planar graph embedded in $\Sigma$ is acyclically 9-list-colourable.
\end{thm}

The bound in Theorem \ref{thm:main} is $\rho_0 = 516(g(\Sigma)-2)$. Note that this implies for example that all graphs embeddable in at least one of the projective plane, torus, and the Klein bottle (which have genus 1,1, and 2, respectively) are acyclically 9-list colourable. 

Table \ref{tab:acyclic} summarizes the best-known bounds for the number of colours required for acyclic (list) colouring planar and locally planar graphs, giving context for the main result of this paper.
\begin{table}[H]
    \centering
    \renewcommand{\arraystretch}{1.3}
    \begin{tabular}{|>{\centering}p{5.5cm}!{\vrule width 2pt}>{\centering}p{3cm}|>{\centering\arraybackslash} p{4.5cm}|}
        \hline
        \backslashbox[6cm]{Graph Class}{Acyclic Chromatic\\ Numbers} &\makebox[3em]{Ordinary}&\makebox[3em]{List} \\
        \noalign{\hrule height 2pt}
        Planar & $= 5$ \hskip0.7cm \cite{borodin1979acyclic5} & \hskip-.5cm$\leq 7$ \hskip1.5cm \cite{borodin2002acycliclist7}  \\
        \hline
        Locally Planar & \hskip0.1cm $\leq 7$ \hskip0.6cm \cite{kawarabayashi2010star} &  \hskip.5cm $\mathbf{\leq 9}$ \hskip0.6cm \textbf{(this paper)}\\
        \hline
    \end{tabular}
    \caption{The best known bounds on the acyclic chromatic and acyclic list chromatic numbers of planar and locally planar graphs. No bound on the acyclic list chromatic number of locally planar graphs was previously known.}
    \label{tab:acyclic}
\end{table}

As ours is the first bound for acyclic list colouring locally planar graphs, we contrast it with a similar result for acyclic colouring; and since the edge-width bound for acyclic 7-colouring locally planar graphs in the paper of Kawarabayashi and Mohar \cite{kawarabayashi2010star} is not given, we compare the bound in Theorem \ref{thm:main} to that obtained by Mohar in \cite{mohar2005acyclic}: Mohar showed an analogous result for acyclic 8-colouring locally planar graphs with $\rho_0:= O(g^32^g)$ (though as mentioned in \cite{mohar2005acyclic}, the author's priority was giving clear, elementary arguments perhaps at the cost of the edge-width bound obtained). 

Given these previous results, we suspect that our list size is not optimal.

\begin{conj}\label{conj:k}
    Let $G$ be a graph embedded in a surface $\Sigma$. There exists $k<9$ and $\rho:= \rho(g(\Sigma))$ such that every $\rho$-locally planar graph embedded in $\Sigma$ is acyclically $k$-list-colourable.
\end{conj}

It seems likely that the methods used in this paper may be used to attain $k = 8$ in Conjecture \ref{conj:k}; however, the analysis would likely be much more lengthy and complex. It is not clear what the optimal value of $k$ is: perhaps $k=7$, which would match the previous results for acyclic list colouring planar graphs and acyclic colouring locally planar graphs, is achievable.

Moreover, despite the improvement from exponential to linear in the genus of the surface, we suspect the edge-width bound in Theorem \ref{thm:main} is not optimal.

\begin{conj}
Let $k$ be an integer such that there exists a value $\rho_0 := \rho_0(g(\Sigma))$ such that every $\rho_0$-locally planar graph embedded in $\Sigma$ is acyclically $k$-list-colourable. There exists $\rho:= O(\log g(\Sigma))$ such that every $\rho$-locally planar graph embedded in $\Sigma$ is acyclically $k$-list-colourable.
\end{conj}

Since (see Theorem V.4.1 in \cite{bollob1978notdef}) this edge-width bound is asymptotically best possible for even ordinary colouring, it follows that it is also optimal for acyclic list colouring. Since there exist planar graphs that are not 4-acyclic colourable, we have moreover that $k \geq 5$.   It seems likely that the same methods used in this paper could be used to show $k \leq 8$, though this would require much more lengthy and careful analysis.

We prove Theorem \ref{thm:main} via a \emph{discharging} argument, the method used to prove many of the results regarding acyclic colouring planar graphs mentioned earlier (see e.g. \cite{borodin1979acyclic5,borodin2002acycliclist7}). However, the standard discharging methods for proving colouring results are insufficient on their own in the context of acyclically list-colouring locally planar graphs, for reasons we elaborate on in Section \ref{sec:overview}. As such, we introduce a novel concept, \emph{weighted edge-width}, which, when used in conjunction with discharging, provides a means to prove our result. This idea is also explained in Section \ref{sec:overview}. 

For those unfamiliar with discharging, we provide a brief history and explanation of the technique. Along with the results mentioned previously, discharging was also the technique used to prove the Four Colour Theorem; in fact, the Four Colour Theorem immensely popularized this now famous technique, which was pioneered by Wernicke \cite{wernicke1904discharging} in the early 1900s and systematized and publicized later by Heesch \cite{heesch1969untersuchungen} in the 1960s. Broadly speaking, proving a theorem via discharging is a two-part process. We begin by assuming for a contradiction that a counterexample to our theorem exists. In one part of the process (the discharging phase), we assign a numerical quantity (called \emph{charge}) to structures in our counterexample in such a way that the total sum of the charges is known. Next, we move the charge around neighbouring structures in the graph via a set of instructions (the \emph{discharging rules}), preserving the total charge. Since the sum of the charges is known, we are able to gain insight into the sorts of structures that occur in the graph. We call these structures \emph{unavoidable}. In the second part of the process, we show that every unavoidable configuration does not occur in a minimum counterexample (by showing the structures are \emph{reducible}), thereby leading to a contradiction. For a more complete overview and further history on the method, we refer the reader to \cite{cranston2017discharging}.

\vskip 3mm
\noindent \textbf{Outline of Paper.}
In Section \ref{sec:overview}, we explain why the existing techniques for arguing the reducibility of unavoidable configurations do not translate to the locally planar setup, and thus cannot be used to prove Theorem \ref{thm:main}. We explain our new tools, and give a brief overview of the proof. The proof itself is found in Section \ref{sec:proof}. Section \ref{sec:proof} contains three subsections: Subsection \ref{subsec:topology} is comprised of several useful topological lemmas that will be used extensively in the later reductions. The results of Subsection \ref{subsec:reducibility} show each configuration in a set of configurations is reducible; and finally, Subsection \ref{subsec:discharging} contains the discharging portion of the proof, showing that at least one of the reducible configurations is unavoidable.

\section{Proof Overview}\label{sec:overview}

First, we discuss why the discharging proof of Borodin, Fon-Der-Flaass, Kostochka, Raspaud, and Sopena \cite{borodin2002acycliclist7} used in proving that planar graphs are acyclically list 7-colourable does not translate to the locally planar environment. For instance, in \cite{borodin2002acycliclist7} the authors assume a vertex-minimum counterexample to their theorem is a triangulation. The authors also delete vertices and add edges between others. We cannot assume our graph is a near-triangulation or add edges between non-adjacent vertices, \emph{as adding edges to our embedded graph can decrease its edge-width}. Moreover, in some of their reductions the authors of \cite{borodin2002acycliclist7} use \say{Kempe chain} arguments, arguing by planarity that pairs of bicoloured paths whose colours are disjoint do not cross. We cannot replicate these arguments, as these bicoloured paths can have arbitrary lengths, and therefore no matter the edge-width bound there is still the possibility of these paths \say{crossing} for example via handles.

New ideas are needed to tackle these difficulties. In the reductions used in many acyclic colouring proofs, vertices are deleted and new edges added between others. As we will explain, adding edges seems crucial to proving that the reducible configurations are indeed acyclically colourable. We give a toy example: in the non-acyclic setting, when proving colouring results by induction, it suffices to delete a \emph{low-degree vertex}, obtain a colouring of the rest of the graph via induction, and then extend the colouring by choosing a colour for the deleted vertex that is not used by its neighbours. As long as the number of colours available is more than the degree of the deleted vertex, the colouring will extend.  For acyclic colouring, when extending the colouring to the deleted vertex we also need to avoid creating bicoloured cycles. Thus \emph{even a vertex $v$ of degree as low as two} does not permit an inductive argument via deletion alone: if two neighbours of the deleted vertex are identically coloured, then we risk creating a bicoloured cycle when extending the colouring to $v$ no matter the choice of colour for the deleted vertex. Thus we would have to avoid not only the colours of neighbours of $v$, but also colours used in its second neighbourhood, which can be prohibitively large and rule out all colour options for $v$. In order to get around this issue, we add edges between neighbours of the deleted vertex to ensure they receive different colours, and therefore that no bicoloured cycle is created when extending the colouring to $v$.  The innovations in our proof hinge on the following observation:  often, these new added edges are added \emph{only to ensure} their endpoints receive different colours (just like in an ordinary vertex colouring argument), rather than simultaneously to ensure endpoints receive different colours and no bicoloured cycles run through these edges. 

This observation motivates partitioning the edge-set of our graph $G$ into two types of edges, each enforcing distinct colouring restrictions. (Note that there is nothing in principle that prevents us from defining yet more edge-subsets with more distinct colouring restrictions and corresponding different edge-lengths: thus we expect this general idea to be useful in other colouring problems.) We define a set $E_1 \subseteq E(G)$ of edges whose endpoints receive distinct colours and which cannot be included in any bicoloured cycle. (These edges function as edges traditionally do in the acyclic colouring framework.) The non-$E_1$-edges enforce fewer restrictions, serving only to ensure that endpoints receive distinct colours. (These edges function as normal edges in ordinary colouring and list colouring.) 

To get around the problem that adding edges judiciously still could result in a decrease of edge-width, we think of these edges as having different lengths: edges in $E_1$ have length $1$, and edges not in $E_1$ have length $t > 1$ (for some suitably chosen fixed $t \in \mathbb{N}$ that arises naturally in the required reductions). For the right choice of $t$, we are able to argue using a set of useful topological lemmas (see Subsection \ref{subsec:topology}) that carefully adding non-$E_1$-edges to our embedded graph does not decrease its edge-width (see Definition \ref{def:ew}).

A summary of our main innovation follows.
\begin{itemize}[\ding{228}]
\item We define \textbf{two types of edges with different colouring restrictions and different lengths.} Throughout our reductions, we add new, long edges to a subgraph of our graph. The length of the new edges ensures the edge-width of the graph does not decrease, and the presence of the new edges allows us to restrict the colouring obtained via induction, thus ensuring the colouring extends to the whole graph.
\item It is important to note that there is nothing in principle that prevents us from defining yet more edge-subsets with more distinct colouring restrictions and corresponding different edge-lengths: thus we expect this general idea to be useful in other colouring problems.
\end{itemize}

We prove Theorem \ref{thm:main} via a more technical (equivalent) theorem (Theorem \ref{thm:technical}) that speaks of edge-width in terms of these edge-lengths. To state this technical theorem, we require the following definitions.
\begin{definition}\label{def:ledgewidth}
Let $G$ be a graph, $k$ a positive integer, and $L$, a $k$-list assignment for $G$. Let $E_1 \subseteq E(G)$. An \emph{$E_1$-cycle} in $G$ is a cycle $C \subseteq G$ with $E(C) \subseteq E_1$. An \emph{$E_1$-acyclic $L$-colouring} of $G$ is a proper $L$-colouring $\varphi$ with the property that $G$ does not contain an $E_1$-cycle $C$ with $|\{\varphi(v): v \in V(C)\}| = 2$. We say $G$ is \emph{$E_1$-acyclic $k$-list-colourable} if it has an $E_1$-acyclic $L'$-colouring for every $k$-list assignment $L'$. 
\end{definition}

We highlight again that this differs from the standard definition of acyclic colouring, which does not allow for \emph{any} bichromatic cycles: in our setup, we allow bichromatic cycles in $G$ so long as they contain at least one edge from from $E(G)\setminus E_1$.

\begin{definition}\label{def:ew}
Let $(G, \Sigma)$ be an embedded graph, let $\mathcal{C}$ be the set of non-contractible cycles of $G$, let $t$ be a positive integer, and let $E_1 \subseteq E(G)$. We define 
$$
{\rm ew}_t(G, E_1) : = \min_{C \in \mathcal{C}} \left( |E_1 \cap E(C)| + t \times |E(C) \setminus E_1| \right).
$$
If $\Sigma$ is the sphere, we define ${\rm ew}_t(G,E_1)$ as being infinite. 
\end{definition}
Note that when $t = 1$, we recover the standard edge-width definition. For our analysis, it suffices to fix $t = 2$.

The more technical version of our main theorem is given below.

\begin{thm}\label{thm:technical}
Let $\Sigma$ be a surface with genus $g$, let $\varepsilon$ be as in Lemma \ref{lemma:discharging}, and let $\rho := \frac{12\cdot(g-2)}{\varepsilon}$. Let $(G, \Sigma)$ be an embedded graph, and $E_1 \subseteq E(G)$. If $\ew(G, E_1) \geq \rho$, then $G$ is $E_1$-acyclic $9$-list-colourable.
\end{thm}

We recover the statement of Theorem \ref{thm:main} by taking $E_1 = E(G)$. (Note the theorems are nearly equivalent: with $\rho_0$ as in Theorem \ref{thm:main}, we recover Theorem \ref{thm:technical} (with less specific edge-width bound) from Theorem \ref{thm:main} by setting $\rho = 2 \cdot \rho_0$.)

Throughout the paper, we will use the following notation.

\begin{definition}
    Given a positive integer $k$, we refer to a vertex of degree $k$, at most $k$, and at least $k$ as a $k$-vertex, a $k^-$-vertex, and a $k^+$-vertex, respectively. Analogously, given an embedded graph $G$, we refer to a face with boundary of length $k$, at most $k$, and at least $k$ as a $k$-face, a $k^-$-face, and a $k^+$-face, respectively.
\end{definition}

\begin{definition}
    Let $G$ be a graph and $v\in V(G)$. We define $N_{E_1}(v)$ as the set $\{u\in N(v) ~:~ vu\in E_1\}$. We also define the \emph{second neighbourhood of $v$} as the set of vertices at distance exactly two from $v$ in $G$.
\end{definition}

\section{Proof of Theorem \ref{thm:technical}}\label{sec:proof}

This section contains the proof of Theorem \ref{thm:technical}. The section is divided into three subsections which contain required tools for the proof. The proof itself is found below. Subsection \ref{subsec:topology} contains a series of useful topological arguments that justify the careful addition of non-$E_1$-edges in our main colouring reductions. These reductions are found in Subsection \ref{subsec:reducibility}, wherein we show that a minimum counterexample to our main theorem does not contain certain substructures. Finally, the discharging portion of the proof (showing at least one of these these substructures must occur, thereby disproving the existence of a minimum counterexample) is found in Subsection \ref{subsec:discharging}.

\begin{proof}[Proof of Theorem \ref{thm:technical}]
Suppose not, and let $G$ be a counterexample that minimizes $v(G)$ and subject to that, maximizes $e(G)$. If $G$ does not contain a non-contractible cycle, then $G$ is planar and hence $G$ is acyclically 9-list colourable by Theorem \ref{thm:borodin}. Thus $G$ contains a non-contractible cycle, and hence $v(G) \geq \frac{6(g-2)}{\varepsilon}$.  

By Lemma \ref{lemma:no3-}, $G$ does not contain a $3^-$-vertex. By Lemma \ref{lemma:no4adjto8^-}, $G$ does not contain a 4-vertex adjacent to an $8^-$-vertex. By Lemma \ref{lemma:no5adj67}, $G$ does not contain $5$-vertex adjacent to both a $6^-$-vertex $v_1$ and a $7^-$-vertex $v_2$ distinct from $v_1$. Finally, by Lemma \ref{lemma:tri6tri6}, $G$ does not contain a  $6$-vertex incident with only $3$-faces adjacent only to $6$-vertices that are themselves only incident with $3$-faces. Yet by Lemma \ref{lemma:discharging}, if $\Sigma$ is a surface of genus $g$ and $G$ is a graph embedded in $\Sigma$ where $v(G) > \frac{6(g-2)}{\varepsilon}$, then $G$ contains one of the configurations listed above. This is a contradiction, and hence a counterexample to Theorem \ref{thm:technical} does not exist. 
\end{proof}

\subsection{Useful Topological Lemmas}\label{subsec:topology}

In this subsection, we prove two lemmas regarding the addition of non-$E_1$-edges to an embedded graph. 

Our first lemma will be used to show that in a minimum counterexample to Theorem \ref{thm:technical}, the faces surrounding low-degree vertices are mostly triangles. This will be helpful in performing colouring reductions in Lemmas \ref{lemma:atleast4E1}, \ref{lemma:no4adjto8^-}, \ref{lemma:4E1 5adj7^-}, \ref{lemma:no5adj67}, and \ref{lemma:tri6tri6}, as it ensures adjacent low-degree vertices have neighbours in common, and therefore that not too many colours are used when colouring their neighbourhoods. 

\begin{lemma}\label{lemma:e1v}
Let $(G, \Sigma)$ be an embedded graph, and let $E_1 \subseteq E(G)$.   If $uv$ and $wv$ are edges in $E_1$ that are cofacial in $(G,\Sigma)$ and $uw \not \in E(G)$, then there exists an embedding of $G+uw$ in $\Sigma$ where $uw$ is embedded in a common face of $uv$ and $wv$ and such that $\ew(G+uw, E_1) 
\geq\ew(G, E_1)$. 
\end{lemma}
\begin{proof}
    Let $(G+uw, \Sigma)$ be the embedded graph obtained from $(G, \Sigma)$ by embedding the edge $uw$ in a face whose boundary contains both $uv$ and $vw$ in such a way that $uvwu$ bounds an open disk. We claim $\ew(G+uw, E_1) = \ew(G, E_1)$. To see this, suppose not, and let $C$ be a non-contractible cycle in $(G+uw, \Sigma)$ with $\ew(C, E_1 \cap E(C)) < \ew(G, E_1)$. Since $\ew(G, E_1) > \ew(G+uw, E_1)$, it follows that $uw \in E(C)$. Note that $C \neq uvwu$, since by our choice of embedding of $uw$, the cycle $uvwu$ is contractible in $\Sigma$. In what follows, let $\Delta$ be the closed disk bounded by the cycle $uwvu$ in $(G+uw,\Sigma)$.

    First suppose $v \not \in V(C)$. Let $C'$ be the embedded graph obtained from $C$ by deleting $uw$ and adding the path $uvw$. Note that since $\{uv,vw\} \subseteq E_1$ and $uw \not \in E_1$, it follows that $\ew(C', E_1 \cap E(C')) \leq \ew(C, E_1 \cap E(C))$. Thus since $C' \subseteq G$, we have that $C'$ is contractible. Let $\Delta'$ be a closed disk in $\Sigma$ with boundary $C'$.   If $\Delta \subseteq \Delta'$, then $\Delta' \setminus \Delta$ is a disk with boundary $C$, contradicting that $C$ is non-contractible. If $\Delta \not \subseteq \Delta'$, then $\Delta \cup \Delta'$ is a disk with boundary $C$, again contradicting the fact that $C$ is non-contractible.

    Thus $v \in V(C)$. Let $C_u$ be the subpath of $C-uw$ from $u$ to $v$ together with the edge $uv$, and let $C_w$ be the subpath of $C-uw$ together with the edge $vw$. Note that each of $C_u$ and $C_w$ is either a cycle, or is isomorphic to $K_2$; and since $C \neq uvwu$, at most one of $C_u$ and $C_w$ is isomorphic to $K_2$.

    First suppose that both $C_u$ and $C_w$ are cycles. Since $uw \not \in E_1$ and $vw \in E_1$ and $uw \not \in E(C_w)$, we have that $\ew(C_w, E_1  \cap E(C_w)) < \ew(C, E_1\cap E(C))$. Similarly, $\ew(C_u, E_1 \cap E(C_u)) < \ew(C, E_1 \cap E(C))$. Since $C_u$ and $C_w$ are cycles in $G$, it follows that $C_u$ and $C_w$ are contractible. Let $\Delta_u$ and $\Delta_w$ be the closed disks in $\Sigma$ with boundary $C_u$ and $C_w$, respectively. If $\Delta_w \subseteq \Delta_u$, then also $\Delta \subseteq \Delta_u$ and hence $\Delta_u \setminus (\Delta \cup \Delta_w)$ is a disk with boundary $C$, contradicting that $C$ is non-contractible. 
    
    Hence we may assume by symmetry that $\Delta_w \not \subseteq \Delta_u$ and $\Delta_u \not \subseteq \Delta_w$. Note that if $\Delta_u \cap \Delta_w \neq \{v\}$, then the boundary of $\Delta_u \cap \Delta_w$ contains at least two shared points in the boundary of $\Delta_u$ (given by $C_u$) and the boundary of $\Delta_w$ (given by $C_w$). This is a contradiction, since $C_w \cap C_u = v$. Thus we may assume $\Delta_u \cap \Delta_v = \{v\}$. But then $\Delta_u \cup \Delta \cup \Delta_v$ is a disk with boundary $C$, again a contradiction. 

    Finally, suppose without loss of generality that $C_u$ is isomorphic to $K_2$, and hence $C_u = uv$. Recall that that since $C \neq uvwu$, we have that $C_w$ is a cycle. Again, note that $\ew(C, E_1\cap E(C)) > \ew(C_w, E_1 \cap E(C_w))$ since in this case $C$ is obtained from $C_w$ by deleting $vw$ and adding the path $vuw$, and $wu \not \in E_1$.  Since $C_w$ is a cycle in $G$, it is contractible and thus bounds a disk; let $\Delta_w$ be the closed disk with boundary $C_w$. If $\Delta \subseteq \Delta_w$, then $\Delta_w \setminus \Delta$ is a disk with boundary $C$, contradicting that $C$ is non-contractible. If $\Delta \not \subseteq \Delta_w$, then $\Delta_w \cup \Delta$ is a disk with boundary $C$, again contradicting that $C$ is non-contractible.
\end{proof}

The following lemma allows us to add specific non-$E_1$-edges when performing our reductions in Lemmas \ref{lemma:atleast4E1}, \ref{lemma:no4adjto8^-}, \ref{lemma:4E1 5adj7^-}, \ref{lemma:no5adj67}, and \ref{lemma:tri6tri6}.

\begin{lemma}\label{lem:starlemma}
    Let $(G, \Sigma)$ be an embedded graph, and $E_1 \subseteq E(G)$. Let $v \in V(G)$, and let $\Delta$ be a closed disk in $(G, \Sigma)$ whose boundary intersects $G$ only in vertices in $N(v)$ and whose interior contains only $v$ and the edges of $\delta(v)$. Suppose $H$ is a graph on the vertex-set $N_{E_1}(v)$  embedded in $\Delta-v$ where $E(H) \cap E(G) = \emptyset$. If $(G',\Sigma)$ is the graph obtained from $(G-v,\Sigma)$ by embedding $H$ in $\Delta-v$ then $\ew(G', E_1 \setminus \delta(v)) \geq \ew(G,E_1)$. 
\end{lemma}
\begin{proof}
Let $(G,H)$ be a counterexample chosen to minimize $e(H)$ and subject to that, to minimize $\deg_G(v)$. Note that $\deg(v) \neq 0$, as otherwise the statement trivially holds since deleting vertices does not decrease edge-width. We claim $V(H) = N(v)$ and $\delta(H) \geq 1$. To see this, suppose not, and let $u$ be either a vertex in $N(v) \setminus V(H)$, or a vertex in $H$ with $\deg_H(u) = 0$.  By the minimality of $(G,H)$, we have that $(G-uv,H)$ is not a counterexample. Since deleting edges does not decrease edge-width, this implies $(G,H)$ is not a counterexample\textemdash a contradiction.  Hence $\delta(H) \geq 1$, and so $e(H) \geq 1$.

 We claim $E(H)$ contains an edge between two consecutive vertices $u,w$ in the boundary of $\Delta$. To see this, let $C$ be the cycle on the vertex-set $V(H)=N(v)$ formed by adding edges between vertices that are consecutive in the cyclic ordering around $v$, and let $H'$ be the graph obtained from $C$ by adding the edges $E(H) \setminus E(C)$.  Note that $H'$ is outerplanar, and hence since outerplanar graphs are 2-degenerate, $H'$ has minimum degree at most 2. Since $\delta(C) = 2$, it follows that there exists a vertex $u \in V(H')$ with degree exactly 2. Since $\delta(H) \geq 1$, this implies $u$ is incident with an edge $uw \in E(H) \cap E(C)$. Hence $uw$ is an edge in $E(H)$ between two consecutive vertices in $C$, the boundary of $\Delta$, and so $u$ and $w$ are cofacial in $G$. By Lemma \ref{lemma:e1v}, there exists an embedding of $G+uw$ with $uw$ in the common face of $uv$ and $wv$  such that $\ew(G+uw, E_1) \geq \ew(G,E_1)$. Let $\Delta^+$ be a closed disk in $(G+uw, \Sigma)$ whose boundary intersects $G$ only in vertices of $N(v)$ and whose interior contains only $v$ and the edges of $\delta(v)$. Let $\Delta' = \Delta^+-v$ and let $H-uw$ be embedded in $\Delta'$. By the minimality of $e(H)$, the graph $(G^+, \Sigma)$ obtained from $(G+uw,\Sigma)$ by replacing $\Delta^+$ by $\Delta'$ satisfies $\ew(G^+, E_1\setminus \delta(v)) \geq \ew(G+uw, E_1)$, a contradiction since $G^+$ is  an embedded graph obtained from $(G-v, \Sigma)$ by embedding $H$ in $\Delta - v$.

\end{proof}
\subsection{Reducible Configurations}\label{subsec:reducibility}

A \emph{counterexample} is a graph, edge-set, and list assignment triple $(G, E_1,L)$ that satisfy the hypotheses but not the conclusion of Theorem \ref{thm:technical}. A \emph{minimum} counterexample is a counterexample chosen to minimize $v(G)$, and subject to that, to maximize $e(G)$. In what follows, let $(G,E_1,L)$ be a minimum counterexample. Note that we may assume by possibly deleting colours from lists that $|L(v)| =9$ for each $v \in V(G)$. We will require the following definitions.

\begin{definition}\label{def:triangular}
    A \emph{triangular vertex} in an embedded graph is a vertex incident only with 3-faces.
\end{definition}

\begin{definition}\label{def:delta}
    Let $H$ be a graph, and $S \subseteq V(H)$. We denote by $\delta(S)$ the set of edges in $H$ with exactly one endpoint in $S$. When $|S| = 1$ (say $S = \{v\}$), we write $\delta(v)$ instead of $\delta(\{v\})$.
\end{definition}

\begin{definition}\label{def:colournotation}
Let $H$ be a graph. Given a (not necessarily proper) colouring $\phi$ of  $H$ and subset $S$ of $V(H)$, we define $\phi(S):=\{\phi(x): x \in X\}$. We say a subgraph $H' \subseteq H$ is \emph{bicoloured with respect to $\phi$} if $|\phi(V(H'))| = 2$. 
\end{definition}
Throughout, given an $L$-colouring $\phi$ of $G$, we are mainly concerned with cycles $C \subseteq G$ that are bicoloured with respect to $\phi$ and have $E(C) \subseteq E_1$.
\begin{definition}
    Let $G$ be a graph, $G'\subseteq G$, $L$ be a list assignment for $G$, $E_1'\subseteq E(G')$, and $E_1 \subseteq E(G)$ such that $E_1'\subseteq E_1$. Suppose that $\phi'$ is an $E_1'$-acyclic $L$-colouring of $G'$. We say \emph{$\phi'$ extends to $G$} if there exists an $E_1$-acyclic $L$-colouring $\phi$ of $G$ that agrees with $\phi'$ on $V(G')$. We call $\phi$ an \emph{extension of $\phi'$}.
\end{definition}

Note that when we extend a colouring $\phi$ to a larger domain since the extension agrees with $\phi$ on its domain, we use the symbol $\phi$ to represent both the original and the extended colouring.  
 
Given $v \in V(G)$, the primary way we contradict the existence of our minimum counterexample $(G,E_1,L)$ is by extending an $(E_1\setminus\delta(v))$-acyclic $L$-colouring $\phi$ of $G\setminus \{v\}$ to $G$. To ensure that the extension of $\phi$ to $G$ is indeed a proper colouring we choose $\phi(v)\notin \phi(N(v))$. To ensure that $G$ contains no $E_1$-cycle that is bicoloured with respect to $\phi$, it is sufficient to ensure $\phi(v)$ does not appear in the set of colours used in the second neighbourhood of $v$. However,  in our reductions, the set of colours that appear in the first and second neighbourhoods of $v$ often exceeds the size of $L(v)$. To get around this, we argue that it suffices to colour $v$ to avoid $\phi(N(v))$ and only \emph{a subset} of the colours that appear in the second neighbourhood of $v$. In particular: let $X$ be the set of pairs  $\{u_1,u_2\}\subseteq N_{E_1}(v)$ with $\phi(u_1)=\phi(u_2)$. If $G$ contains an $E_1$-cycle $C$ that is bicoloured with respect to $\phi$ and where $v \in V(C)$, then $C$ also contains a pair of vertices in $X$. Therefore, instead of avoiding all colours in the second neighbourhood of $v$, it suffices to avoid the colours of the neighbours of at least one vertex from each pair of vertices in $X$. This notion is formalized in the following lemma.

\begin{lemma}\label{lemma:extendcolouring}
    Let $G$ be a graph, $E_1\subseteq E(G)$ and $L$ be a list assignment for $G$. Let $v\in V(G)$ and $E_1' = E_1 \setminus \delta(v)$. Suppose that $\phi$ is an $E_1'$-acyclic $L$-colouring of $G\setminus \{v\}$.
    Let $S\subseteq N_{E_1}(v)$ be such that for all $\{u_1,u_2\} \subseteq N_{E_1}(v)$ such that $\phi(u_1)=\phi(u_2)$, we have that $S\cap \{u_1,u_2\} \ne \emptyset$. Let $N_S : = \bigcup_{u\in S}~(N_{E_1}(u)\setminus\{v\})$. If $L(v)\setminus (\phi(N(v))\cup\phi(N_S))$ is non-empty, then $\phi$ extends to an $E_1$-acyclic $L$-colouring of $G$.

\end{lemma}

\begin{proof}
     We have that $L(v)\setminus (\phi(N(v))\cup\phi(N_S))$ is non-empty therefore we extend $\phi$ to $G$ by choosing $\phi(v)\in L(v)\setminus (\phi(N(v))\cup\phi(N_S))$. Suppose that the extension of $\phi$ to $G$ is not a valid $E_1$-acyclic $L$-colouring of $G$. Since $\phi(v)\notin \phi(N(v))$, we have that $\phi$ is a proper colouring and therefore $G$ contains a cycle $C$ which is bicoloured with respect to $\phi$.

    If $v\notin V(C)$ then $C\subseteq G\setminus \{v\}$, contradicting that $\phi$ is an $E_1'$-acyclic $L$-colouring of $G\setminus \{v\}$. Thus $v\in V(C)$ and therefore there exist two neighbours $u_1$ and $u_2$ of $v$ such that $\{u_1,u_2\}\subseteq V(C)$. By definition of $S$, at least one of $u_1$ and $u_2$ is in $S$; we may assume that $u_1\in S$. Let $u\neq v$ be a vertex adjacent to $u_1$ in $C$. Note that since $C$ is an $E_1$-cycle, we have that $u_1u\in E_1$. Since $C$ is bicoloured with respect to $\phi$, $\phi(v)=\phi(u)$. However this contradicts the fact that $\phi(v)\neq \phi(u)$ since $u\in N_S$.
\end{proof}

Often in the reductions, we divide our proofs into cases based on the quantity $|\phi(N(v))|$, where $v \in V(G)$ and $\phi$ is an $(E_1\setminus \delta(v))$-acyclic $L$-colouring of $G\setminus \{v\}$. We assume that $\phi$ is chosen to maximize $|\phi(N(v))|$, and therefore if we create an $(E_1\setminus \delta(v))$-acyclic $L$-colouring $\phi'$ with $|\phi'(N(v))|>|\phi(N(v))|$ we arrive at a contradiction. Let $u\in N(v)$, and assume that all colours in $\phi(N_{E_1}(u))$ are distinct. If $L(v)\setminus \phi(N(v)\cup N(u))$ is non-empty, then there exists an alternate colour for $u$ which gives us an $(E_1\setminus\delta(v))$-acyclic $L$-colouring with $|\phi'(N(v))|>|\phi(N(v))|$. We then can conclude that $|\phi(N_{E_1}(u))|<|N_{E_1}(u)|$, which will help us in applying Lemma \ref{lemma:extendcolouring}, since it will reduce the size of $S$. A more general version of this recolouring argument is captured in the following lemma.

\begin{lemma}\label{lemma:recolour}
    Let $G$ be a graph, $E_1\subseteq E(G)$, $L$ be a list assignment for $G$ and $v\in V(G)$. Suppose that $\phi$ is an $E_1$-acyclic $L$-colouring of $G$. Let $C$ be a set of colours such that $\phi(N(v)\cup \{v\})\subseteq C$. If $|\phi(N_{E_1}(v))| = |N_{E_1}(v)|$ and $L(v)\setminus C$ is non-empty, then there exists an $E_1$-acyclic $L$-colouring $\phi'\neq \phi$ such that $\phi'(x)= \phi(x)$ for all $x\in V(G)\setminus\{v\}$ and $\phi'(v)\notin C$.
\end{lemma}

\begin{proof}

    Let $\phi'$ be a colouring of $G'$ obtained by defining $\phi'(x)=\phi(x)$ for all $x\in V(G)\setminus \{v\}$ and $\phi'(v)\in L(v)\setminus C$. Note that $\phi'$ is a proper colouring as $\phi'(v)\notin \phi(N(v))\subseteq C$. In addition, $G$ contains no bicoloured cycle with respect to $\phi'$ since every $E_1$-cycle containing $v$ also contains two vertices $\{v_1,v_2\}\subseteq N_{E_1}(v)$ which have distinct colours different from $\phi'(v)$ since $|\phi(N_{E_1}(v))|=|N_{E_1}(v)|$.
    
\end{proof}

To facilitate the reductions in this section, we will use the following.

\begin{lemma}\label{lemma:atleast4E1}
    If $v \in V(G)$ is an $8^-$-vertex, then there are at least four $E_1$-edges incident with $v$.
\end{lemma}
\begin{proof}

Suppose towards a contradiction that there exists $v\in V(G)$ with $\deg(v)\leq 8$ and $|N_{E_1}(v)|\leq 3$. 

Let $H$ be the complete graph with vertex-set $N_{E_1}(v)$. Let $G'$ be the graph with vertex-set $V(G)\setminus \{v\}$ and edge-set $(E(G)\cup E(H))\setminus \delta(v)$. Let $E_1':= E_1 \setminus \delta(v)$. By Lemma \ref{lem:starlemma}, $G'$ has an embedding such that $\ew(G',E_1')\geq \ew(G,E_1)$. Since $G$ is a minimum counterexample to Theorem \ref{thm:technical} and $v(G')<v(G)$, the graph $G'$ has an $E_1'$-acyclic $L$-colouring $\phi$.

Since $\deg_G(v)\leq 8$, we have that $|L(v)\setminus \phi(N(v))| \geq |L(v)|- |\phi(N(v))| \ge 9 - 8 = 1$ and therefore $L(v)\setminus \phi(N(v))$ is non-empty. By Lemma \ref{lemma:extendcolouring} with $S:= \emptyset$, $\phi$ extends to an $E_1$-acyclic $L$-colouring of $G$, a contradiction.

\end{proof}
As an easy consequence of Lemma \ref{lemma:e1v} and the edge-maximality of $G$, we have the following.
\begin{lemma}\label{lemma:addE2edges}
    $G$ does not contain vertices $u,v,w$ where $\{uv, vw\} \subset E_1$, where $uv$ and $vw$ are cofacial, and where $uw \not \in E(G)$.
\end{lemma}
\begin{proof}

Let $\{uv,vw\}\subset E_1$ and suppose by way of contradiction that $uw\notin E(G)$. By Lemma \ref{lemma:e1v} the graph $G'$ obtained from $G$ by adding the non $E_1$-edge $uw$ to $G$ satisfies $\ew(G',E_1) \geq \ew(G,E_1)$.  This contradicts the fact that $G$ is edge-maximal.

\end{proof}

The lemmas below establish the reducibility of the unavoidable configurations listed in Lemma \ref{lemma:discharging}.

\begin{lemma}\label{lemma:no3-}
    $G$ does not contain a $3^-$-vertex.
\end{lemma}
\begin{proof}

By Lemma \ref{lemma:atleast4E1}, every $8^-$-vertex is incident to four $E_1$-edges. Therefore, if $v\in V(G)$ has $\deg(v)\leq 8$, it also has $\deg(v)\geq 4$, so no $3^-$-vertex exists in $V(G)$.

\end{proof}

\begin{figure}
    \centering
    \begin{tikzpicture}[scale=1.5]

    \node[label={[label distance=0.1cm]235:$v$},fill=gray!50] (v) at (0,0) {};
    \foreach \s in {1,...,4}
    {
        \node (v\s) at ({360/4 + 360/4 * (\s - 1)}:1cm) {};
        \ifnum \s < 4 
            \node[label={[label distance=0.1cm]{360/4 + 360/4 * (\s - 1)}:$v_{\s}$}] at (v\s) {};
        \else
            \node[label={[label distance=0.1cm]270:$u$}] at (v\s) {};
        \fi
    }
    
    \foreach \s [evaluate=\s as \t using {int(\s+3)}, evaluate=\s as \u using {\s/2}] in {-2,-1,...,2} 
    {
        \node[label = {[label distance=0.1cm]{45/5 * (\s)}:$u_{\t}$}] (u\t) at ({45/5 * (\s)}:2.2cm) {};
    }

    \foreach \s [evaluate=\s as \t using {int(mod(\s,4)+1)}] in {1,2,3,4}
    {
        \draw (v) edge[deleted] (v\s);
        \draw (v\s) edge[E1] (v\t);
    }
    \foreach \s in {1,2,...,5}
    {
        \draw (v4) edge[E1] (u\s);
    }
    \draw (v1) edge[E2, bend left] (v3);    
    \end{tikzpicture}
    \caption{The structure of the graph in Lemma \ref{lemma:no4adjto8^-}. A $4$-vertex $v$ is adjacent to a $8^-$-vertex $u$. The set $N_u$ is labelled as $\{u_1,\dots,u_5\}$.}
    \label{fig:4-8^-}
\end{figure}
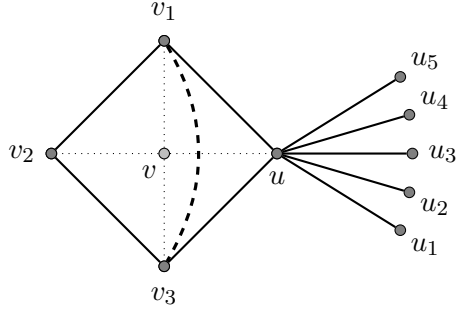

Our second reducible configuration (referred to in Lemma \ref{lemma:no4adjto8^-}) is illustrated in Figure \ref{fig:4-8^-}.
\begin{lemma} \label{lemma:no4adjto8^-}
$G$ does not contain a $4$-vertex adjacent to an $8^-$-vertex.
\end{lemma}
\begin{proof}

Suppose not: that is, suppose there exists a $4$-vertex $v$ adjacent to an $8^-$-vertex $u$. By Lemma \ref{lemma:no3-}, $G$ does not contain a $3^-$-vertex and hence $\deg(u) \geq 4$. Since $\deg(v) = 4$, by Lemma \ref{lemma:atleast4E1} every vertex is incident with at least four $E_1$-edges and hence every edge incident with $v$ is in $E_1$. Let $u,v_1,v_2,v_3$ be the neighbours of $v$ listed in cyclic order.  Since $\{vu, vv_1, vv_3\} \subseteq E_1$, by Lemma \ref{lemma:addE2edges} we find that $\{v_1,v_3\}\subseteq N(u)$. Let $N_u := N_{E_1}(u)\setminus (N(v)\cup \{v\})$; note that since $\deg(u)\leq 8$, we have that $|N_u|\leq 5$.

Let $H$ be the graph with vertex-set $N_{E_1}(v)$ and edge-set $\{v_1v_3\}$.  Let $G'$ be the graph with vertex-set $V(G)\setminus \{v\}$ and edge-set $(E(G)\cup E(H))\setminus \delta(v)$. Let $E_1':= E_1 \setminus \delta(v)$. By Lemma \ref{lem:starlemma}, $G'$ has an embedding such that $\ew(G',E_1')\geq \ew(G,E_1)$. Since $G$ is a minimum counterexample to Theorem \ref{thm:technical} and $v(G')<v(G)$, the graph $G'$ has an $E_1'$-acyclic $L$-colouring $\phi$.

First suppose that $|\phi(N(v))|= 4$. Since $|L(v)\setminus \phi(N(v))|\geq |L(v)| - |\phi(N(v))| =  9 - 4 =5$, we have that $L(v)\setminus \phi(N(v))$ is non-empty. By Lemma \ref{lemma:extendcolouring} with $S:=\emptyset$, we have that $\phi$ is an $E_1$-acyclic $L$-colouring of $G$, contradicting that $G$ is a counterexample to Theorem \ref{thm:technical}. Thus we may assume $|\phi(N(v))|\leq 3$. Since $\{v_1v_3, v_1v_2, v_2v_3, uv_1, uv_3\} \subseteq E(G')$, it follows that $\phi(v_2) = \phi(u)$ and $|\phi(N(v))|= 3$. 

Let $S:=\{u\}$. Since $|\phi(N(v))|= 3$ and $|N_u|\leq 5$, we have $|L(v)\setminus (N(v) \cup N_u)|\geq |L(v)|-|\phi(N(v))|-|N_u|\geq 9-3-5=1$ and therefore $L(v)\setminus \phi(N(v)\cup N_u)$ is non-empty. By Lemma  \ref{lemma:extendcolouring} with $S$ as previously defined, we have that $\phi$ extends to an $E_1$-acyclic $L$-colouring of $G$, a contradiction.

\end{proof}

\begin{figure}
    \centering
    \begin{tikzpicture}[scale=1.5]

    \node[label={[label distance=0.1cm]265:$v$},fill=gray!50] (v) at (0,0) {};
    \foreach \s in {1,...,5}
    {
        \node (v\s) at ({360/5 + 360/5 * (\s - 1)}:1cm) {};
        \ifnum \s < 5 
            \node[label={[label distance=0.1cm]{360/5 + 360/5 * (\s - 1)}:$v_{\s}$}] at (v\s) {};
        \else
            \node[label={[label distance=0.1cm]270:$u$}] at (v\s) {};
        \fi
    }
    
    \foreach \s [evaluate=\s as \t using {int(\s+3)}, evaluate=\s as \u using {\s/2}] in {-2,-1,...,2} 
    {
        \node[label = {[label distance=0.1cm]{45/5 * (\s)}:$u_{\t}$}] (u\t) at ({45/5 * (\s)}:2.2cm) {};
    }

    \foreach \s in {1,2,3,4,5}
    {
        \ifnum\s=1
            \draw (v) edge[deletedE2] (v\s);
        \else
            \draw (v) edge[deleted] (v\s);
        \fi
    }
    \foreach \s [evaluate=\s as \t using {int(mod(\s,5)+1)}] in {2,3,4}
    {
        \draw (v\s) edge[E1] (v\t);
    }
    \foreach \s in {1,2,3,4,5}
    {
        \draw (v5) edge[E1] (u\s);
    }
    \draw (v2) edge[E2, bend right=15] (v4);
    \draw (v2) edge[E2, bend left] (v5);
    
    \end{tikzpicture}
    \qquad
    \begin{tikzpicture}[scale=1.5]

    \node[label={[label distance=0.1cm]265:$v$},fill=gray!50] (v) at (0,0) {};
    \foreach \s in {1,...,5}
    {
        \node (v\s) at ({360/5 + 360/5 * (\s - 1)}:1cm) {};
        \ifnum \s < 5 
            \node[label={[label distance=0.1cm]{360/5 + 360/5 * (\s - 1)}:$v_{\s}$}] at (v\s) {};
        \else
            \node[label={[label distance=0.1cm]270:$u$}] at (v\s) {};
        \fi
    }
    
    \foreach \s [evaluate=\s as \t using {int(\s+2.5)}, evaluate=\s as \u using {\s/2}] in {-1.5,-0.5,0.5,1.5} 
    {
        \node[label = {[label distance=0.1cm]{45/4 * (\s)}:$u_{\t}$}] (u\t) at ({45/4 * (\s)}:2.2cm) {};
    }

    \foreach \s in {1,2,3,4,5}
    {
        \ifnum\s=2
            \draw (v) edge[deletedE2] (v\s);
        \else
            \draw (v) edge[deleted] (v\s);
        \fi
    }
    \foreach \s [evaluate=\s as \t using {int(mod(\s,5)+1)}] in {3,4,5}
    {
        \draw (v\s) edge[E1] (v\t);
    }
    \foreach \s in {1,2,3,4}
    {
        \draw (v5) edge[E1] (u\s);
    }
    \draw (v1) edge[E2, bend left=15] (v4);
    \draw (v1) edge[E2, bend right] (v3);
    
    \end{tikzpicture}
    \caption{The structure of the graph in Lemma \ref{lemma:4E1 5adj7^-}. A $5$-vertex $v$ is adjacent to a $7^-$-vertex $u$ such that $vu\in E_1$ and $|E_1\cap \delta(v)|=4$. The case where $vv_1\notin E_1$ is on the left, and the case where $vv_2 \notin E_1$ is on the right. The set $N_u$ is labelled as $\{u_1,\dots,u_5\}$ (left) or $\{u_1,\dots,u_4\}$ (right).}
    \label{fig:4E15adj7^-}
\end{figure}
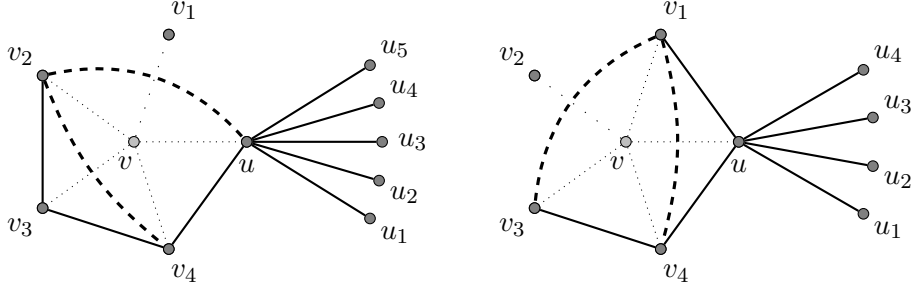

\begin{lemma}\label{lemma:4E1 5adj7^-}
    $G$ does not contain a $5$-vertex $v$ adjacent to a $7^-$-vertex $u$, such that $vu\in E_1$ and $|\delta(v)\cap E_1|=4$.
\end{lemma}
\begin{proof}

    Suppose not. By Lemma \ref{lemma:no3-}, $G$ does not contain a $3^-$-vertex and hence $\deg(u)\geq 4$. Let $u,v_1,v_2,v_3,v_4$ be the neighbours of $v$ listed in cyclic order and set $v_0:= u$. Let $i \in \{1,2,3,4\}$ be the index with $v_i\in N(v)\setminus N_{E_1}(v)$. By symmetry, we may assume that $i\in \{1,2\}$. Since $vv_4$ and $vu$ are cofacial and $\{vv_4,vu\}\subseteq E_1$, by Lemma \ref{lemma:addE2edges}, we have that $v_4\in N(u)$. Let $N_u := N_{E_1}(u)\setminus (N(v)\cup \{v\})$; note that since $\deg(u)\leq 7$, we have $|N_u|\leq 5$. Also by Lemma \ref{lemma:addE2edges}, we have that consecutive vertices in $N_{E_1}(v)$, with respect to the cyclic ordering of $N(v)$, are adjacent. See Figure \ref{fig:4E15adj7^-} for an illustration for when $i=1$ (left) and $i=2$ (right).

    Let $H$ be the graph with vertex-set $N_{E_1}(v)$ and edge-set $\{v_{3-i}v_4,v_{i-1}v_{i+1}\}$. Let $G'$ be the graph with vertex-set $V(G)\setminus \{v\}$ and edge-set $(E(G)\cup E(H))\setminus \delta(v)$. Let $E_1':= E_1 \setminus \delta(v)$. By Lemma \ref{lem:starlemma}, $G'$ has an embedding such that $\ew(G',E_1')\geq \ew(G,E_1)$. Since $G$ is a minimum counterexample to Theorem \ref{thm:technical} and $v(G')<v(G)$, the graph $G'$ has an $E_1'$-acyclic $L$-colouring. Let $\phi$ be an $E_1'$-acyclic $L$-colouring of $G'$ that maximizes $|\phi(N_{E_1}(v))|$.

    First suppose that $|\phi(N_{E_1}(v))|=4$. Since $|L(v)\setminus \phi(N(v))|\geq |L(v)|-|\phi(N(v))|\geq 9-5=4$, we have that $L(v)\setminus \phi(N(v))$ is non-empty. By Lemma \ref{lemma:extendcolouring} with $S:=\emptyset$, we have that $\phi$ extends to an $E_1$-acyclic $L$-colouring of $G$, a contradiction. Thus we assume that $|\phi(N_{E_1}(v))|\leq 3$. Since $\{uv_{3-i},v_{3-i}v_3,v_3v_4,uv_4,v_{3-i}v_4\}\subseteq E(G')$, it follows that $\phi(v_3)=\phi(u)$ and $\{\phi(v_{3-i}),\phi(v_3),\phi(v_4)\}$ are all distinct colours. Note then that $|\phi(N_{E_1}(v))|=3$. 
    
    Let $S :=\{u\}$. First suppose that all vertices in $N(u)\setminus \{v\}$ receive different colours under $\phi$. If $uv_1\in E(G)$, then $|N_u|\leq 4$ and therefore $|L(u)\setminus \phi(N_{E_1}(v)\cup N(u))|=|L(u)\setminus \phi(N(v)\cup N_u)|\geq |L(u)|-|\phi(N(v))|-|\phi(N_u)|\leq 9-4-4=1$. On the other hand, if $uv_1\notin E(G)$, then since $uv \in E_1$, by Lemma \ref{lemma:addE2edges} $v_1\notin N_{E_1}(v)$. Therefore, $|L(u)\setminus \phi(N_{E_1}(v)\cup N(u))|= |L(u)\setminus \phi(N_{E_1}(v)\cup N_u)|\geq |L(u)|-|\phi(N_{E_1}(v))|-|N_u|\geq 9-3-5=1$. In either case, the set $L(u)\setminus \phi(N_{E_1}(v)\cup N(u))$ is non-empty. By Lemma \ref{lemma:recolour} we obtain a colouring $\phi'$ of $G'$ that agrees with $\phi$ everywhere except for $u$, and where $\phi'(u)\notin \phi(N_{E_1}(v))$. Since $|\phi'(N_{E_1}(v))|=4> |\phi(N_{E_1}(v))|$, we have that $\phi'$ contradicts our choice of $\phi$. We therefore conclude that there is a vertex $u'\in N_u$ such that $\phi(u')\in \phi(N(u)\setminus \{v,u'\})$.

    Since $|L(v)\setminus\phi(N(v)\cup N_u)|\geq |L(v)|-|\phi(N(v))|-|N_u\setminus\{u'\}| \geq 9-4 - 4 = 1$, we have  that $L(v)\setminus \phi(N(v)\cup N_u)$ is non-empty. By Lemma \ref{lemma:extendcolouring} (with $S = \{u\}$) we have that $\phi$ extends to an $E_1$-acyclic $L$-colouring of $G$, a contradiction.

\end{proof}

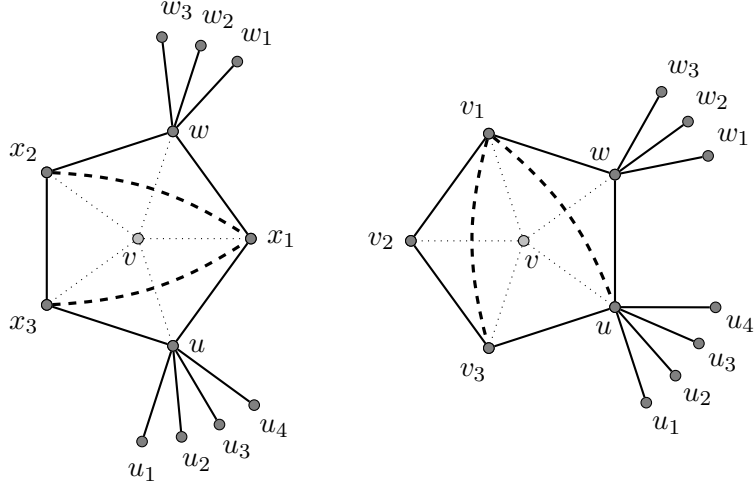
\begin{figure}
    \centering
    \begin{tikzpicture}[scale=1.5]

    \node[label={[label distance=0.1cm]265:$v$},fill=gray!50] (v) at (0,0) {};
    \newcounter{index}
    \setcounter{index}{1}
    \foreach \s in {1,4,2,3,5}
    {
        \node (v\s) at ({360/5 * (\theindex - 1)}:1cm) {};
        \ifnum \s < 4 
            \node[label={[label distance=0.1cm]{360/5 * (\theindex - 1)}:$x_{\s}$}] at (v\s) {};
        \else
            \ifnum \s = 4
                \node[label={[label distance=0.1cm]0:$w$}] at (v\s) {};
            \else
                \node[label={[label distance=0.1cm]0:$u$}] at (v\s) {};
            \fi
        \fi
        \stepcounter{index}
    }
    
    \foreach \s [evaluate=\s as \t using {int(\s+2.5)}, evaluate=\s as \u using {\s/2}] in {-1.5,-0.5,0.5,1.5} 
    {
        \node[label = {[label distance=0.1cm]{(-72)+45/4 * (\s)}:$u_{\t}$}] (u\t) at ({(-72) + 45/4 * (\s)}:1.8cm) {};
        
    }
    \foreach \s [evaluate=\s as \t using {int(\s+2)}, evaluate=\s as \u using {\s/2}] in {-1,0,1} 
    {
        \node[label = {[label distance=0.1cm]{72+45/4 * (\s)}:$w_{\t}$}] (w\t) at ({72 + 45/4 * (\s)}:1.8cm) {};
    }

    \foreach \s in {1,2,3,4,5}
    {
        \draw (v) edge[deleted] (v\s);
    }
    \draw (v1) edge[E1] (v4); 
    \draw (v4) edge[E1] (v2); 
    \draw (v2) edge[E1] (v3); 
    \draw (v3) edge[E1] (v5); 
    \draw (v5) edge[E1] (v1);
    \foreach \s in {1,2,3,4}
    {
        \draw (v5) edge[E1] (u\s);
    }
    \foreach \s in {1,2,3}
    {
        \draw (v4) edge[E1] (w\s);
    }
    \draw (v2) edge[E2, bend left=15] (v1);
    \draw (v3) edge[E2, bend right=15] (v1);
    
    \end{tikzpicture}
    \qquad
    \begin{tikzpicture}[scale=1.5]

    \node[label={[label distance=0.1cm]275:$v$},fill=gray!50] (v) at (0,0) {};
    \foreach \s in {1,...,5}
    {
        \node (v\s) at ({36 + 360/5 + 360/5 * (\s - 1)}:1cm) {};
        \ifnum \s < 4 
            \node[label={[label distance=0.1cm]{36 + 360/5 + 360/5 * (\s - 1)}:$v_{\s}$}] at (v\s) {};
             
        \else
            \ifnum \s = 4
                \node[label={[label distance=0.1cm]255:$u$}] at (v\s) {};
            \else
                \node[label={[label distance=0.1cm]115:$w$}] at (v\s) {};
            \fi
        \fi
    }
    
    \foreach \s [evaluate=\s as \t using {int(\s+2.5)}, evaluate=\s as \u using {\s/2}] in {-1.5,-0.5,0.5,1.5} 
    {
        \node[label = {[label distance=0.1cm]{(-36)+45/4 * (\s)}:$u_{\t}$}] (u\t) at ({(-36) + 45/4 * (\s)}:1.8cm) {};
        
    }
    \foreach \s [evaluate=\s as \t using {int(\s+2)}, evaluate=\s as \u using {\s/2}] in {-1,0,1} 
    {
        
        \node[label = {[label distance=0.1cm]{(36)+45/4 * (\s)}:$w_{\t}$}] (w\t) at ({(36) + 45/4 * (\s)}:1.8cm) {};
    } 
    \node[color=white] at ({270}:2.15cm) {};

    \foreach \s in {1,2,3,4,5}
    {
        \draw (v) edge[deleted] (v\s);
    }
    \foreach \s [evaluate=\s as \t using {int(mod(\s,5)+1)}] in {1,2,3,4,5}
    {
        \draw (v\s) edge[E1] (v\t);
    }
    \foreach \s in {1,2,3,4}
    {
        \draw (v4) edge[E1] (u\s);
    }
    \draw (v1) edge[E2, bend left=15] (v4);
    \draw (v1) edge[E2, bend right=15] (v3);
    \foreach \s in {1,2,3}
    {
        \draw (v5) edge[E1] (w\s);
    }
    
    \end{tikzpicture}
    \caption{The structure of the graph in Lemma \ref{lemma:no5adj67}. A $5$-vertex $v$ is adjacent to a $7^-$-vertex $u$ and $6^-$-vertex $w$ distinct from $u$. The case where  $u$ and $w$ are not consecutive in the cyclic order of $N(v)$ is on the left, and the case where $u$ and $w$ are consecutive is on the right. The sets $N_u$ and $N_w$ are labelled as $\{u_1,\dots,u_4\}$ and $\{u_1,\dots,u_3\}$, respectively.}
    \label{fig:no5adj67}
\end{figure}
\begin{lemma}\label{lemma:no5adj67}
   $G$ does not contain a $5$-vertex adjacent to a both $7^-$-vertex $u$ and a $6^-$-vertex $w$ distinct from $u$.
\end{lemma}

\begin{proof}

    Suppose not, and let $v$ be a $5$-vertex adjacent to a $7^-$-vertex $u$ and a $6^-$-vertex $w$ distinct from $u$. By Lemma \ref{lemma:no3-}, we have that $G$ does not contain a $3^-$-vertex and hence $\deg(u)\geq 4$ and $\deg(w)\geq 4$.  Since $v$ is an $8^-$-vertex, by Lemma \ref{lemma:atleast4E1} we find that  $|\delta(v)\cap E_1|\geq 4$. At least one of $u$ and $w$ is in $N_{E_1}(v)$ and hence by Lemma \ref{lemma:4E1 5adj7^-} we have $|\delta(v)\cap E_1|\neq 4$, and therefore $|\delta(v)\cap E_1| =5$.
    
    We split into cases based on whether or not $u$ and $w$ are consecutive in the cyclic ordering of $N(v)$.  The two cases considered are illustrated in Figure \ref{fig:no5adj67}. When $u$ and $w$ are consecutive, let $u,w,v_1,v_2,v_3$ be a cyclic ordering of $N(v)$. When $u$ and $w$ are not consecutive, let $u,x_1,w,x_2,x_3$ be a cyclic ordering of $N(v)$. By Lemma \ref{lemma:addE2edges} it follows that any two consecutive vertices are adjacent in $G$. Let $N_u := N_{E_1}(u)\setminus (N(v)\cup \{v\})$ and $N_w := N_{E_1}(w)\setminus (N(v)\cup\{v\})$. Since $\deg(u)\leq 7$ and $\deg(w)\leq 6$, it follows that $|N_u|\leq 4$ and $|N_w|\leq 3$.
    
    If $u$ and $w$ are consecutive, let $H$ be the graph with vertex-set $N_{E_1}(v)$ and edge-set $\{v_1v_3,uv_1\}$. Let $G'$ be the graph with vertex-set $V(G)\setminus \{v\}$ and edge-set $(E(G)\cup E(H))\setminus \delta(v)$. Let $E_1':= E_1 \setminus \delta(v)$. By Lemma \ref{lem:starlemma}, $G'$ has an embedding such that  $\ew(G',E_1')\geq \ew(G,E_1)$. Since $G$ is a minimum counterexample to Theorem \ref{thm:technical} and $v(G')<v(G)$, the graph $G'$ has an $E_1'$-acyclic $L$-colouring. Let $\phi$ be an $E_1'$-acyclic $L$-colouring which maximizes $|\phi(N(v))|$.
    
    If $u$ and $w$ are not consecutive, let $H$ be the graph with vertex-set $N_{E_1}(v)$ and edge-set $\{x_1x_2,x_1x_3\}$. Let $G'$ be the graph with vertex-set $V(G)\setminus \{v\}$ and edge-set $(E(G)\cup E(H))\setminus \delta(v)$. Let $E_1':= E_1 \setminus \delta(v)$. By Lemma \ref{lem:starlemma}, $G'$ has an embedding such that $\ew(G',E_1')\geq \ew(G,E_1)$. Since $G$ is a minimum counterexample to Theorem \ref{thm:technical} and $v(G')<v(G)$, the graph $G'$ has an $E_1'$-acyclic $L$-colouring. Let $\phi$ be an $E_1'$-acyclic $L$-colouring which maximizes $|\phi(N(v))|$.

    First suppose that $|\phi(N(v))|=5$. Since $|L(v)\setminus \phi(N(v))|\geq |L(v)|-|\phi(N(v))|=9-5=4$, the set $L(v)\setminus\phi(N(v))$ is non-empty. By Lemma \ref{lemma:extendcolouring} with $S:=\emptyset$, we have that $\phi$ extends to an $E_1$-acyclic $L$-colouring of $G$, a contradiction. Thus we assume that $|\phi(N(v))|\leq 4$.

    Next suppose that $|\phi(N(v))|=4$. Since $|N(v)|=5$, there is exactly one pair of vertices in $N(v)$ with the same colour. Since $G'[N(v)\setminus\{u,w\}]$ is isomorphic to $K_3$, at least one of $u$ or $w$ is in this pair. Let $S := \{z\}$ where $z\in \{u,w\}$ is said vertex. Since $|L(v)\setminus \phi(N(v)\cup N_z)|\geq |L(v)|-|\phi(N(v))|-|N_z|\geq 9-4-4=1$, we have that $L(v)\setminus \phi(N(v)\cup N_z)$ is non-empty. By Lemma \ref{lemma:extendcolouring} with $S$ as defined previously, we have that $\phi$ extends to an $E_1$-acyclic $L$-colouring of $G$, a contradiction. Thus we assume that $|\phi(N(v))|= 3$.

    As before, $G'[N(v)\setminus \{u,w\}]$ is isomorphic to $K_3$ and therefore any pair of vertices in $N(v)$ with the same colour includes at least one of $\{u,w\}$. Let $S:=\{u,w\}$. 
    
    Recall the cyclic order of $N(v)$ when $u$ and $w$ are not consecutive: $u,x_1,w,x_2,x_3$. Since $|\phi(N(v))|=3$, we have $\phi(u)=\phi(x_2)\neq \phi(x_3)=\phi(w)$. Suppose that $|\phi(N(w)\setminus\{v\})|=|N(w)\setminus \{v\}|=5$, and so that $|\phi(N_w)|=3$. Since $|L(w)\setminus \phi(N(v)\cup N_w)|\geq |L(w)|-|\phi(N(v))|-|N_w|\geq 9-3-3=3$, the set $L(w)\setminus \phi(N(v)\cup N_w)$ is non-empty. By Lemma \ref{lemma:recolour} we obtain a colouring $\phi'$ of $G'$ that agrees with $\phi$ everywhere except for $w$, and where $\phi'(w)\notin \phi(N(v))$. Since $|\phi'(N_{E_1}(v))|=4> |\phi(N_{E_1}(v))|$, we have that $\phi'$ contradicts our choice of $\phi$. We therefore conclude that there exists a vertex $w'\in N_w$ such that $\phi(w')\in \phi(N(w)\setminus \{v,w'\})$. A symmetrical argument shows that there exists a vertex $u'\in N_u$ such that $\phi(u)\in \phi(N(u)\setminus \{v,u'\})$.
    
    Since $|L(v)\setminus \phi(N(v)\cup N_u\cup N_w)|\geq |L(v)|-|\phi(N(v))|-|N_u\setminus\{u'\}|-|N_w\setminus \{w'\}| \geq 9-3-3-2=1$, we have that $L(v)\setminus (\phi(N(v)\cup N_u\cup N_w))$ is non-empty. By Lemma \ref{lemma:extendcolouring} (with $S:=\{u,w\}$), $\phi$ extends to an $E_1$-acyclic $L$-colouring of $G$, a contradiction. We therefore conclude that $u$ and $w$ are consecutive in the cyclic order of $N(v)$.

    Recall the cyclic order of $N(v)$ when $u$ and $w$ are consecutive: $u,w,v_1,v_2,v_3$. Since $|\phi(N(v))|=3$ and both $\{v_1,v_2,v_3\}$ and $\{v_1,v_3,u\}$ induce triangles, we have $\phi(v_2)=\phi(u)\neq \phi(w)=\phi(v_3)$. Let $S:=\{u,w\}$. Suppose that $|\phi(N(w)\setminus\{v\})|=|N(w)|=5$, and so that $|\phi(N_w)|=3$. Since $|L(w)\setminus \phi(N(v)\cup N_w)|\geq |L(w)|-|\phi(N(v))|-|N_w|\geq 9-3-3=3$, the set $L(w)\setminus \phi(N(v)\cup N_w)$ is non-empty. By Lemma \ref{lemma:recolour} we obtain a colouring $\phi'$ of $G'$ that agrees with $\phi$ everywhere except for $w$, and where $\phi'(w)\notin \phi(N(v))$. Since $|\phi'(N_{E_1}(v))|=4> |\phi(N_{E_1}(v))|$, we have that $\phi'$ contradicts our choice of $\phi$. We therefore conclude that there exists a vertex $w'\in N_w$ such that $\phi(w')\in \phi(N(w)\setminus \{v,w'\})$. 
    
    Note that we may not use a symmetrical recolouring argument for $u$ in general since $\phi(w)=\phi(v_3)$ and $\{w,v_3\}\subseteq N(u)$. Suppose that $uw\notin E_1$, and furthermore assume that $|\phi(N_u\cup\{v_3\})|=|N_u \cup\{v_3\}|=5$. Since $|L(u)\setminus \phi(N(v)\cup N_u)|\geq |L(u)|-|\phi(N(v))|-|N_u|\geq 9-3-4=2$, the set $L(u)\setminus \phi(N(v)\cup N_u)$ is non-empty. Since $w\notin N_{E_1}(u)$, by Lemma \ref{lemma:recolour} we obtain a colouring $\phi'$ of $G'$ that agrees with $\phi$ everywhere except for $u$, and where $\phi'(u)\notin \phi(N(v))$. Since $|\phi'(N_{E_1}(v))|=4> |\phi(N_{E_1}(v))|$, we have that $\phi'$ contradicts our choice of $\phi$. We therefore conclude that there exists a vertex $u'\in N_u$ such that $\phi(u')\in \phi(N(u)\setminus \{v,u'\})$.

    Since $|L(v)\setminus \phi(N(v)\cup N_u\cup N_w)|\geq |L(v)|-|\phi(N(v))|-|N_u\setminus\{u'\}|-|N_w\setminus \{w'\}| \geq 9-3-3-2=1$, we have that $L(v)\setminus \phi(N(v)\cup N_u\cup N_w)$ is non-empty. By Lemma \ref{lemma:extendcolouring} (with $S:=\{u,w\}$), $\phi$ extends to $E_1$-acyclic $L$-colouring of $G$, a contradiction. We may then assume that $uw\in E_1$.

    Suppose that the face $F$ incident with $uw$, which is not the triangle given by $u,v,w$, is also a triangle. This implies that $|N_u\cap N_w|\geq 1$. Since $|L(v)\setminus \phi(N(v)\cup N_u\cup N_w)|\geq |L(v)|-|N(v)|-|N_w\setminus \{w'\}|-|N_u\setminus N_w|\geq  9-3-2-3=1$, we have that $L(v)\setminus \phi(N(v)\cup N_u\cup N_w)$ is non-empty. By Lemma \ref{lemma:extendcolouring} with $S$ as previously defined, $\phi$ extends to $E_1$-acyclic $L$-colouring of $G$, a contradiction. We then conclude that the face $F$ is not a triangle.

    Since $F$ is not a triangle and $uw\in E_1$, by Lemma \ref{lemma:addE2edges} there exists $u'\in N(u)\setminus N_{E_1}(u)$. Symmetrically, we find $w'\in N(w)\setminus N_{E_1}(w)$. Specifically, $u'$ and $w'$ are the vertices such that $uu'$ and $ww'$ are cofacial with $uw$. Therefore, $|N_u|\leq 3$ and $|N_w|\leq 2$. Since $|L(v)\setminus \phi(N(v)\cup N_u\cup N_w)|\geq |L(v)|-|\phi(N(v))|-|N_u|-|N_w|\geq 9-3-3-2=1$, we have that $L(v)\setminus \phi(N(v)\cup N_u\cup N_w)$ is non-empty. By Lemma \ref{lemma:extendcolouring} (with $S:=\{u,w\}$), $\phi$ extends to an $E_1$-acyclic $L$-colouring of $G$, a contradiction.

\end{proof}

Our final reducible configuration follows.

\begin{figure}
    \centering
    \begin{tikzpicture}[scale=1.5]

    \node[label={[label distance=0.1cm]270:$v$},fill=gray!50] (v) at (0,0) {};
    \foreach \s in {1,...,6}
    {
        \ifnum \s = 3
            \node[label={[label distance=0.15cm]{-10}:$v_{\s}$}] (v\s) at ({-60 + 360/6 * (\s - 1)}:1cm) {};
        \else
            \ifnum \s = 4
                \node[label={[label distance=0.15cm]{190}:$v_{\s}$}] (v\s) at ({-60 + 360/6 * (\s - 1)}:1cm) {};
            \else
                \node[label={[label distance=0.15cm]{270}:$v_{\s}$}] (v\s) at ({-60 + 360/6 * (\s - 1)}:1cm) {};
            \fi
        \fi
        
    }
    \foreach \s in {1,3,5,7,9,11}
    {
        \node (u\s) at ({-90 + 360/6 * (\s - 1)/2}:1.75cm) {};
    }
    \foreach \s in {2,4,6,8,10,12}
    {
        \node (u\s) at ({-60 + 360/6 * (\s - 2)/2}:2cm) {};
    }
    \foreach \s in {1,...,12}
    {
        \node[label={[label distance=0.15cm]{-90 + 360/12 * (\s - 1)}:$u_{\s}$}] at (u\s) {};
    }

    \draw (v) edge[deleted] (v1);
    \draw (v) edge[deleted] (v2);
    \draw (v) edge[deletedE2] (v3);
    \draw (v) edge[deleted] (v4);
    \draw (v) edge[deleted] (v5);
    \draw (v) edge[deletedE2] (v6);
    \foreach \s [evaluate=\s as \t using {int(mod(\s,6)+1)}] in {1,...,6}
    {
        \draw (v\s) edge[E1] (v\t);
        \foreach \i [evaluate=\i as \j using {int(mod(2*\s +\i,12)+1)}] in {-2,-1,0} {
            \draw(v\s) edge[E1] (u\j);
        }
    }
    \foreach \s [evaluate=\s as \t using {int(mod(\s,12)+1)}] in {1,...,12}
    {
        \draw (u\s) edge[E1] (u\t);
    }
    \draw (v1) edge[E2,bend left=45] (v4);
    \draw (v2) edge[E2] (v4);
    
    \end{tikzpicture}
    \qquad
    \begin{tikzpicture}[scale=1.5]

    \node[label={[label distance=0.1cm]270:$v$},fill=gray!50] (v) at (0,0) {};
    \foreach \s in {1,...,6}
    {
        \ifnum \s = 3
            \node[label={[label distance=0.15cm]{-10}:$v_{\s}$}] (v\s) at ({-60 + 360/6 * (\s - 1)}:1cm) {};
        \else
            \ifnum \s = 4
                \node[label={[label distance=0.15cm]{190}:$v_{\s}$}] (v\s) at ({-60 + 360/6 * (\s - 1)}:1cm) {};
            \else
                \node[label={[label distance=0.15cm]{270}:$v_{\s}$}] (v\s) at ({-60 + 360/6 * (\s - 1)}:1cm) {};
            \fi
        \fi
        
    }
    \foreach \s in {1,3,5,7,9,11}
    {
        \node (u\s) at ({-90 + 360/6 * (\s - 1)/2}:1.75cm) {};
    }
    \foreach \s in {2,4,6,8,10,12}
    {
        \node (u\s) at ({-60 + 360/6 * (\s - 2)/2}:2cm) {};
    }
    \foreach \s in {1,...,12}
    {
        \node[label={[label distance=0.15cm]{-90 + 360/12 * (\s - 1)}:$u_{\s}$}] at (u\s) {};
    }
    
    \draw (v) edge[deleted] (v1);
    \draw (v) edge[deleted] (v2);
    \draw (v) edge[deleted] (v3);
    \draw (v) edge[deleted] (v4);
    \draw (v) edge[deleted] (v5);
    \draw (v) edge[deleted] (v6);
    \foreach \s [evaluate=\s as \t using {int(mod(\s,6)+1)}] in {1,...,6}
    {
        \draw (v\s) edge[E1] (v\t);
        \foreach \i [evaluate=\i as \j using {int(mod(2*\s +\i,12)+1)}] in {-2,-1,0} {
            \draw(v\s) edge[E1] (u\j);
        }
    }
    \foreach \s [evaluate=\s as \t using {int(mod(\s,12)+1)}] in {1,...,12}
    {
        \draw (u\s) edge[E1] (u\t);
    }
    \draw (v1) edge[E2] (v3);
    \draw (v1) edge[E2] (v5);
    \draw (v3) edge[E2] (v5);
    
    \end{tikzpicture}
    
    \caption{The structure of the graph in Lemma \ref{lemma:tri6tri6}. A triangular $6$-vertex $v$ is adjacent to only triangular $6$-vertices. On the left we demonstrate one of the cases where $\{vv_1,vv_3,vv_5\}\not\subseteq E_1$ and $\{vv_2,vv_4,vv_6\}\not\subseteq E_1$. The case where $\{vv_1,vv_3,vv_5\}\subseteq E_1$ is on the right. The sets $N_4$, $N_5$, and $N_6$ are labelled as $\{u_7,u_8,u_9\}$, $\{u_9,u_{10},u_{11}\}$, and $\{u_{11},u_{12},u_1\}$, respectively.}
    \label{fig:tri6tri6}
\end{figure}
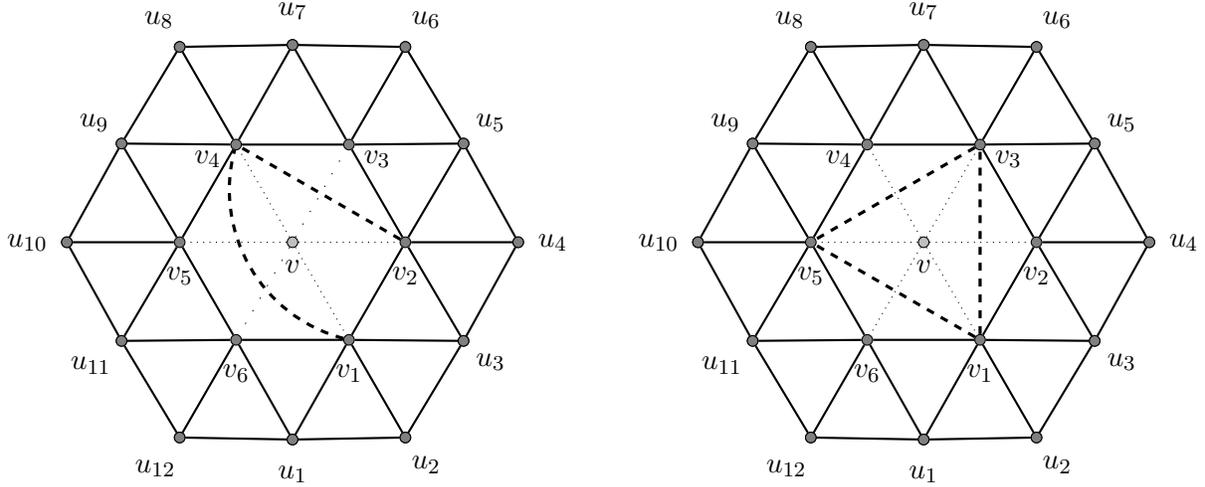

\begin{lemma}\label{lemma:tri6tri6}
    $G$ does not contain a triangular $6$-vertex whose neighbourhood contains only triangular $6$-vertices. 
\end{lemma}
\begin{proof}

    Suppose not, and let $v$ be a triangular $6$-vertex whose neighbourhood contains only triangular $6$-vertices. Let $v_1,\dots,v_6$ be the neighbours of $v$ listed in cyclic order. By Lemma \ref{lemma:atleast4E1}, every $8^-$-vertex is incident with at least four $E_1$-edges and therefore $|\delta(v)\cap E_1|\geq 4$.

    First suppose that $\{vv_1,vv_3,vv_5\}\not\subseteq E_1$ and $\{vv_2,vv_4,vv_6\}\not\subseteq E_1$. We may assume that $vv_1\in E_1$ and $vv_6\notin E_1$. We conclude that $\{vv_2,vv_4\}\subseteq E_1$ and exactly one of $vv_3$ and $vv_5$ is in $E_1$. Let $u$ be the vertex in $\{v_3,v_5\}\cap N_{E_1}(v)$.

    Let $H$ be the graph with vertex-set $N_{E_1}(v)$ and edge-set $\{v_1v_4,v_2v_4\}$. See Figure \ref{fig:tri6tri6} (left) for an illustration of this case. Let $G'$ be the graph with vertex-set $V(G)\setminus \{v\}$ and edge-set $(E(G)\cup E(H))\setminus \delta(v)$. Let $E_1':= E_1 \setminus \delta(v)$. By Lemma \ref{lem:starlemma}, $G'$ has an embedding such that  $\ew(G',E_1')\geq \ew(G,E_1)$. Since $G$ is a minimum counterexample to Theorem \ref{thm:technical} and $v(G')<v(G)$, the graph $G'$ has an $E_1'$-acyclic $L$-colouring. Let $\phi$ be an $E_1'$-acyclic $L$-colouring of $G'$.
        
    First suppose that $|\phi(\{v_1,v_2,v_4,u\})| = 4$. Since $|L(v)\setminus \phi(N(v))|\geq |L(v)|-|N(v)|=9-6=3$, we have that $L(v)\setminus \phi(N(v))$ is non-empty. By Lemma \ref{lemma:extendcolouring} with $S:=\emptyset$, we have that $\phi$ extends to an $E_1$-acyclic $L$-colouring of $G$, a contradiction. Thus $|\phi(\{v_1,v_2,v_4,u\})|\leq 3$. Since $v$ is a triangular vertex $v_1v_2\in E(G)$. As well, since $\{v_1v_2,v_1v_4,v_2v_4\}\subseteq E(G')$, we have that $\{v_1,v_2,v_4\}$ induces a triangle in $G'$ and therefore $|\phi(\{v_1,v_2,v_4,u\})|=3$.
    
    Since $v$ is a triangular vertex, $\{v_3v_4,v_4v_5\}\subseteq E(G)$. Since $\{v_1v_2,v_1v_4,v_2v_4\}\subseteq E(G')$, it follows that $\phi(u)$ and exactly one of $\phi(v_1)$ and $\phi(v_2)$ are equal. Let $S:=\{u\}$. Since $|\phi(\{v_1,v_2,v_4,u\})|\leq 3$, we have that $|\phi(N(v))|\leq 5$. Let $N_u := N_{E_1}(u)\setminus (N(v)\cup \{v\})$. Since $u$ is a triangular $6$-vertex by assumption, $|\phi(N_u)|\leq |N_u| \leq 3$. Since $|L(v)\setminus \phi(N(v)\cup N_u)|\geq |L(v)|-|\phi(N(v))|-|\phi(N_u)| \geq 9-5-3 = 1$, the set $L(v)\setminus \phi(N(v)\cup N_u)$ is non-empty. By Lemma \ref{lemma:extendcolouring} with $S$ as previously defined, we have that $\phi$ extends to an $E_1$-acyclic $L$-colouring of $G$, a contradiction. 
    
    Thus we may assume without loss of generality that $\{vv_1,vv_3,vv_5\}\subseteq E_1$. See Figure \ref{fig:tri6tri6} (right) for an illustration of this case. Let $H$ be the graph with vertex-set $N_{E_1}(v)$ and edge-set $\{v_1v_3,v_1v_5,v_3v_5\}$. Let $G'$ be the graph with vertex-set $V(G)\setminus \{v\}$ and edge-set $(E(G)\cup E(H))\setminus \delta(v)$. Let $E_1':= E_1 \setminus \delta(v)$. By Lemma \ref{lem:starlemma}, $G'$ has an embedding such that $\ew(G',E_1')\geq \ew(G,E_1)$. Since $G$ is a minimum counterexample to Theorem \ref{thm:technical} and $v(G')<v(G)$, the graph $G'$ has an $E_1'$-acyclic $L$-colouring. Let $\phi$ be an $E_1'$-acyclic $L$-colouring which maximizes $|\phi(N(v))|$. 
    
    First suppose that $|\phi(N(v))|=6$. Since $|L(v)\setminus \phi(N(v))|\geq |L(v)|-|\phi(N(v))|=9-6=3$, we have that $L(v)\setminus \phi(N(v))$ is non-empty. By Lemma \ref{lemma:extendcolouring} with $S:=\emptyset$, we have that $\phi$ extends to an $E_1$-acyclic $L$-colouring of $G$, a contradiction. Thus we may assume that $|\phi(N(v))|\leq 5$.
    
    Suppose that $|\phi(N(v))|=5$. Then there are exactly two vertices $\{u,w\}\subseteq N(v)$ such that $\phi(u)=\phi(w)$. Let $S:=\{u\}$, and let $N_u := N_{E_1}(u)\setminus (N(v)\cup \{v\})$. Since $u$ is a triangular $6$-vertex, $|\phi(N_u)|\leq |N_u|\leq 3$. Since $|L(v)\setminus \phi(N(v)\cup N_u)|\geq |L(v)|- |\phi(N(v))|-|\phi(N_u)| = 9-5-3=1$, the set $L(v)\setminus \phi(N(v)\cup N_u)$ is non-empty. By Lemma \ref{lemma:extendcolouring} with $S$ as previously defined, we have that $\phi$ extends to an $E_1$-acyclic $L$-colouring of $G$, a contradiction. Thus we may assume that $|\phi(N(v))|\leq 4$.
    
    For the remainder of the proof we use the notation $N_i$ to denote $N_{E_1}(v_i)\setminus (N(v)\cup \{v\})$ for $i\in [6]$. Next suppose that $|\phi(N(v))|=4$. Since $\{v_1v_3,v_1v_5,v_3v_5\}\subseteq E(G')$, there exists a vertex $u\in \{v_2,v_4,v_6\}$ such that $|\phi(\{v_1,v_3,v_5,u\})|=4$. Up to relabelling the vertices, we may assume that $u= v_2$. Let $S := \{v_4,v_6\}$.

    Recall that $v_4$ is a triangular $6$-vertex by assumption. Suppose that $|\phi(N(v_4)\setminus \{v\})|=5$, and so that $|\phi(N_4)|=3$. Since $|L(v_4)\setminus (\phi(N(v)\cup N_{4}))|\geq |L(v_4)|-|\phi(N(v))|-|N_{4}|\geq 9-4-3 = 2$, the set $L(v_4)\setminus \phi(N(v)\cup N_{4})$ is non-empty. By Lemma \ref{lemma:recolour} we obtain a colouring $\phi'$ of $G'$ that agrees with $\phi$ everywhere except for $v_4$, and where $\phi'(v_4)\notin \phi(N(v))$. Since $|\phi'(N_{E_1}(v))|=5> |\phi(N_{E_1}(v))|$, we have that $\phi'$ contradicts our choice of $\phi$. We therefore conclude that there exists a vertex $v_4'\in N_{4}$ such that $\phi(v_4')\in \phi(N(v_4)\setminus \{v,v_4'\})$. A symmetrical argument implies that there exists a vertex $v_6'\in N_6$ such that $\phi(v_6')\in \phi(N(v_6)\setminus \{v,v_6'\})$.

    Since $|L(v)\setminus(\phi(N(v)\cup N_4 \cup N_6))|\geq |L(v)|-|\phi(N(v))|- |N_4\setminus \{v_4'\}| -|N_6\setminus \{v_6'\}| \geq 9-4 - 2 - 2= 1$, we have that $L(v)\setminus \phi(N(v)\cup N_4\cup N_6)$ is non-empty. By Lemma \ref{lemma:extendcolouring} (with $S:=\{v_4,v_6\}$), we have that $\phi$ extends to an $E_1$-acyclic $L$-colouring of $G$, a contradiction. We conclude that $|\phi(N(v))| = 3$.

    Since $\{v_1v_3,v_1v_5,v_3v_5\}\subseteq E(G')$, we have $\phi(v_i)=\phi(v_{i+3})$ for $i\in \{1,2,3\}$. Let $S:=\{v_4,v_5,v_6\}$. As before, using Lemma \ref{lemma:recolour} we conclude that  for $i\in \{4,6\}$ there exists a vertex $v_i'\in N_i$ such that $\phi(v_i')\in \phi(N(v)\cup (N_i)\setminus\{v_i'\})$.

    Since $v_5$ is a triangular vertex, $\{v_4v_5, v_5v_6\}\subseteq E(G)$ and so  $|N_5\setminus (N_4\cup N_6)|\leq 1$. Thus $|\phi(N(v)\cup N_4\cup N_5 \cup N_6)|\leq |\phi(N(v))| + |N_4\setminus \{v_4'\}| + |N_6\setminus \{v_6'\}| + |N_5\setminus (N_4\cup N_6)|\leq 3+2+2+1=8$. Since $|L(v)|=9$, we have that $L(v)\setminus (\phi(N(v)\cup N_4 \cup N_5 \cup N_6))$ is non-empty. By Lemma \ref{lemma:extendcolouring} (with $S:=\{v_4,v_5,v_6\}$), we have that $\phi$ extends to an $E_1$-acyclic $L$-colouring of $G$, a contradiction.

\end{proof}

\subsection{Discharging}\label{subsec:discharging}

In this subsection, we prove our discharging lemma, which is used to show that every locally planar embedded graph contains at least one member of a set of configurations (listed in Lemma \ref{lemma:discharging}). As shown in Subsection \ref{subsec:reducibility}, each of these configurations is reducible, and therefore does not appear in a minimum counterexample to Theorem \ref{thm:technical}.

\begin{lemma}\label{lemma:discharging}
Let $\varepsilon$ be a real number with $0 < \varepsilon \leq \frac{1}{43}$. If $G$ is a graph embedded in a surface of genus $g$ with $v(G) > \frac{6(g-2)}{\varepsilon}$, then $G$ contains at least one of the following configurations:
\begin{enumerate}[(i)]
    \item a $3^-$ vertex,
    \item a $4$-vertex adjacent to an $8^-$-vertex,
    \item a $5$-vertex adjacent to both a $6^-$-vertex $v_1$ and a $7^-$-vertex $v_2$ distinct from $v_1$, or 
    \item a triangular $6$-vertex adjacent only to triangular $6$-vertices.
\end{enumerate}

\end{lemma}

    \begin{proof}
    Suppose not. We assign an initial charge of $ch_0(v) = \deg(v)-6$ to each $v \in V(G)$, and $ch_0(f) = 2(\deg(f)-3)$ to each face $f$ in the set $F$ of faces of the embedding of $G$. Using Euler's formula for graphs embedded in surfaces, 
    \begin{align}
        \sum_{v \in V(G)} ch_0(v) + \sum_{f \in F}ch_0(f) &= 6(e(G)-v(G) - |F|) \nonumber\\
        &\leq 6(g-2) \label{eq:sumofcharges}.
    \end{align} 
    We discharge via three rules. For $i \in [3]$, we denote by $ch_i$ the charge after applying rule Ri. We aim to show that $ch_3(v) \geq \varepsilon$ and $ch_3(f) \geq 0$ for each vertex $v \in V(G)$ and face $f \in F$. This will imply that the sum of all charges is at least $\varepsilon \cdot v(G)$; and since $v(G) > \frac{6(g-2)}{\varepsilon}$ by assumption, this will imply further that the sum of all charges is greater than $6(g-2)$, contradicting Equation \ref{eq:sumofcharges}. Note that $ch_0(f) \geq 0$ for each face $f \in F$, and $ch_0(v) \geq \varepsilon$ for each $7^+$-vertex. Our discharging rules (which follow below) therefore revolve around sending extra charge to $6^-$-vertices. 
\paragraph{Discharging Rules}

    \begin{itemize}
        \item[R1:] Every $4^+$-face $f$ splits its charge $ch_0(f)$ evenly between its incident $6^-$-vertices. 
        \item[R2:] Every $7^+$-vertex $v$ sends $\frac{ch_0(v)-\varepsilon}{\deg(v)}$ charge to each of its neighbouring vertices. Charge sent via this rule to $7^+$-vertices is instead redirected:  $\frac{1}{2}\frac{ch_0(v)-\varepsilon}{\deg(v)}$ is sent to each of the next $6^-$-vertex clockwise and anticlockwise from $u$ in the cyclic ordering of $N(v)$. 
        \item[R3:] Every $6$-vertex $v$ that is incident with a $4^+$-face or adjacent to a $7^+$-vertex splits $(ch_2(v)-\varepsilon)$ evenly between its neighbouring triangular 6-vertices that are not adjacent to a $7^+$-vertex.
    \end{itemize}

    In R2, if charge sent from $u$ to $w$ is instead redirected to $v$, we say \emph{$u$ sends $v$ charge redirected from $w$}.

    We state a useful observation regarding R2.

    \begin{obs}\label{redir}
        Let $v \in V(G)$, and let $u$ and $w$ be consecutive $7^+$-vertices in the cyclic order of $N(v)$. Let $c_u,c_w>0$ be the amount of charge sent directly (not via redirection) to $v$ from $u$ and $w$, respectively. If there exists a triangular face incident with both $vu$ and $vw$, then $u$ sends $v$ at least $\frac{c_u}{2}$ charge redirected from $w$, and $w$ sends $v$ at least $\frac{c_w}{2}$ charge redirected from $u$.
    \end{obs}
    Note that if $vu$ and $vw$ are incident with the same triangular face, then $uw$ is an edge.

    First note that since $ch_0(f)  \geq 0$ for each face $f \in F$ and faces only send charge via R1, it follows from R1 that $ch_3(f) \geq 0$ for each face $f$. Next, we claim every $7^+$-vertex $v\in V(G)$ has $ch_3(v) \geq \varepsilon$: this follows easily from the facts that $ch_0(v) = \deg(v) - 6$, and $v$ sends charge to at most $\deg(v)$ neighbours via R2. Hence $ch_3(v) \geq (\deg(v)- 6) - \deg(v)\cdot\left( \frac{ch_0(v)-\varepsilon}{\deg(v)} \right)= \varepsilon$, as desired.
    
    To show all vertices in $G$ have final charge at least $\varepsilon$, it remains only to consider the $6^-$-vertices. Since $G$ contains no $3^-$-vertex by assumption, we limit our attentions to the $4$-, $5$-, and $6$-vertices. For the remainder of the analysis, we break into cases.

    \paragraph{Case 1: $v$ is a 4-vertex.} Since $v$ is a 4-vertex, $ch_0(v) = -2$. By assumption, $v$ is not adjacent to an $8^-$-vertex, and hence all neighbours of $v$ are $9^+$-vertices. Note that if $f$ is a $4^+$-face incident with $v$, then two of the $9^+$-neighbours of $v$ are in the boundary walk of $f$, and hence we have that $f$ sends $v$ at least $\frac{ch_0(f)}{\deg(f)-2} = \frac{2(\deg(f)-3)}{\deg(f)-2}$  via R1. This is at least 1, since $\deg(f) \geq 4$. We break into further cases depending on the number of $4^+$-faces incident with $v$.
    
    First suppose $v$ is incident with at least two $4^+$-faces. Each incident $4^+$-face sends $v$ at least 1 charge via R1, and each neighbouring $9^+$-vertex sends at least $\frac{3-\varepsilon}{9}$ charge via R2, and hence $ch_3(v) \geq ch_0(v) + 2(1) + 4\left(\frac{3-\varepsilon}{9}\right) = 4\left(\frac{3-\varepsilon}{9}\right) > \varepsilon$, where the last inequality follows since $\varepsilon < \frac{12}{13}$.
    
    Next suppose $v$ is incident with exactly one $4^+$-face, and therefore is incident with three triangles. The incident $4^+$-face sends $v$ at least $1$ charge via R1. By Observation \ref{redir} we conclude that $ch_3(v)\geq ch_0(v) + 1 + 4\left(\frac{3-\varepsilon}{9} \right) + 3\left(\frac{1}{2}\cdot\frac{3-\varepsilon}{9} + \frac{1}{2}\cdot\frac{3-\varepsilon}{9}\right) = -1 + 7\left(\frac{3-\varepsilon}{9}\right)> \varepsilon$, where the last inequality follows since $\varepsilon < \frac{3}{4}$. 

    Thus we may assume $v$ is a triangular vertex. By Observation \ref{redir} we conclude that $ch_3(v)\geq ch_0(v) + 4\left(\frac{3-\varepsilon}{9} \right) + 4\left(\frac{1}{2}\cdot\frac{3-\varepsilon}{9} + \frac{1}{2}\cdot\frac{3-\varepsilon}{9}\right) = -2 + 8\left(\frac{3-\varepsilon}{9}\right) > \varepsilon$ where the last inequality follows since $\varepsilon \leq \frac{6}{17}$.
    
   \paragraph{Case 2: $v$ is a 5-vertex.} Since $v$ is a 5-vertex, $ch_0(v) = -1$. By assumption, $v$ is not adjacent to a $7^-$-vertex $u$ and a $6^-$ vertex distinct from $u$, and hence $v$ is adjacent to at most one $6^-$-vertex. Moreover, if $v$ is adjacent to exactly one $6^-$-vertex, all other neighbours of $v$ are $8^+$-vertices. Since each face incident with $v$ is also incident with at least one $7^+$-vertex, each $4^+$-face $f$ incident with $v$ sends $v$ at least $\frac{ch_0(f)}{\deg(f)-1} = \frac{2(\deg(f)-3)}{\deg(f)-1} \geq \frac{2}{3}$ charge via R1. Hence if $v$ is incident with at least two $4^+$-faces, then $ch_2(v) \geq -1 + 2 \cdot \frac{2}{3} \geq \varepsilon$ since $\varepsilon \leq \frac{1}{3}$. Thus we may assume $v$ is incident with at most one $4^+$-face.

   First suppose that $v$ is incident with exactly one $4^+$-face $f$ (and hence exactly four triangles). Then $v$ receives $\frac{2}{3}$ charge from $f$ via R1. If $v$ is adjacent to a $6^-$-vertex, then since the remaining neighbours of $v$ are $8^+$-vertices, at least two faces satisfy the premises of Observation \ref{redir} and so $ch_3(v)\geq -1 + \frac{2}{3} + 4\left(\frac{2-\varepsilon}{8}\right) + 2 \left(\frac{1}{2}\cdot\frac{2-\varepsilon}{8} + \frac{1}{2}\cdot\frac{2-\varepsilon}{8}\right)= -\frac{1}{3}+ 6\left(\frac{2-\varepsilon}{8}\right) $. Note this is at least $\varepsilon$ since $\varepsilon \leq \frac{2}{3}$.  Meanwhile if every vertex in $N(v)$ is a $7^+$-vertex, all four triangular faces satisfy the premises of Observation \ref{redir}, and therefore we have $ch_3(v)\geq -1 + \frac{2}{3} + 5\left(\frac{1-\varepsilon}{7}\right) + 4\left(\frac{1}{2}\cdot\frac{1-\varepsilon}{7} + \frac{1}{2}\cdot\frac{1-\varepsilon}{7}\right) = -\frac{1}{3} + 9\left(\frac{1-\varepsilon}{7}\right)> \varepsilon$ since $\varepsilon \leq \frac{5}{12}$. 

   Lastly suppose that $v$ is a triangular vertex. If $v$ is adjacent to a $6^-$-vertex, then since it is adjacent to at most one $6^-$-vertex by assumption, exactly three faces satisfy the premises of Observation \ref{redir} and so $ch_3(v)\geq -1 + 4\left(\frac{2-\varepsilon}{8}\right) + 3 \left(\frac{1}{2}\cdot\frac{2-\varepsilon}{8} + \frac{1}{2}\cdot\frac{2-\varepsilon}{8}\right)= -1 + 7\left(\frac{2-\varepsilon}{8}\right) $ Note this is at least $\varepsilon$ since $\varepsilon \leq \frac{2}{5}$.  Meanwhile if every vertex in $N(v)$ is a $7^+$-vertex, then every face incident with $v$ satisfies Observation \ref{redir} and therefore we have $ch_3(v)\geq -1 + 5\left(\frac{1-\varepsilon}{7}\right) + 5\left(\frac{1}{2}\cdot\frac{1-\varepsilon}{7} + \frac{1}{2}\cdot\frac{1-\varepsilon}{7}\right) = -1 + 10\left(\frac{1-\varepsilon}{7}\right)> \varepsilon$ since $\varepsilon \leq \frac{3}{17}$. 

   \paragraph{Case 3: $v$ is a 6-vertex.} If $v$ receives charge via R1, then it is incident with a $4^+$-face, and hence $v$ receives at least $\frac{2(\deg(f)-3)}{\deg(f)} \geq \frac{2(4-3)}{4} = \frac{1}{2}> \varepsilon$ charge via R1. Next, $v$ gives $ch_2(v)-\varepsilon$ to its neighbours via R3, leaving $\varepsilon$ charge for $v$, as desired. We may assume then that $v$ does not receive charge via R1, and hence $v$ is triangular.

   Suppose then that $v$ receives charge via R2. Then $v$ is  adjacent to a $7^+$-vertex, and hence $v$ receives at least $\frac{1-\varepsilon}{7}> \varepsilon $ charge via R2. Note this is at least $\varepsilon$ since $\varepsilon \geq \frac{1}{8}$. Once again, $v$ gives $ch_2(v)-\varepsilon$ to its neighbours via R3, leaving $\varepsilon$ charge for $v$, as desired. We may assume then that $v$ does not receive charge via R2, and hence is not adjacent to any $7^+$-vertex.

   Suppose then that $v$ receives charge via R3. Then it is adjacent to a $6$-vertex $w$ which is adjacent to a $7^+$-vertex or incident with a $4^+$-face. Hence $w$ is a 6-vertex that received charge via either R1 or R2, and so as argued above, we see that $ch_2(w)\geq \frac{1-\varepsilon}{7}$ if $w$ is adjacent to a $7^+$-vertex or $ch_2(w)\geq \frac{1}{2}$ if $w$ is incident with a $4^+$-face. In the former case, since $w$ is adjacent to at least one $7^+$-vertex, it sends charge to at most five neighbours via R3. Therefore $v$ receives at least $\frac{1}{5}\left(\frac{1-\varepsilon}{7}- \varepsilon\right)\geq \varepsilon$ from $w$ via R3, which is at least $\varepsilon$ since $\varepsilon \leq \frac{1}{43}$. Meanwhile, in the latter case we have that $v$ receives $\frac{1}{6}\left(\frac{1}{2}-\varepsilon\right)\geq \varepsilon$ from $w$, which is at least $\varepsilon$ since $\varepsilon \leq \frac{1}{14}$.

   Thus we may assume $v$ is adjacent to only triangular $6$-vertices and $5^-$-vertices. Since $G$ does not contain a $4$-vertex adjacent to an $8^-$-vertex, all vertices in $N(v)$ are 5-vertices or triangular 6-vertices. Since $G$ does not contain a 5-vertex adjacent to two $6^-$-vertices, it follows further that $v$ is a triangular $6$-vertex adjacent only to triangular 6-vertices. This is a contradiction.

\end{proof}

\bibliographystyle{siam}
\bibliography{bibliog}
\end{document}